\documentclass[a4paper,11pt,reqno]{customart}

\usepackage[utf8x]{inputenc}
\usepackage[margin=1in]{geometry}
\usepackage{verbatim}
\usepackage{color}
\usepackage[table]{xcolor}
\usepackage{listings}
\usepackage{amssymb,amsfonts,amsthm,amsmath}
\usepackage{url}
\usepackage{enumerate}
\usepackage{todonotes}
\usepackage[all,cmtip]{xy}
\usepackage{multirow}
\usepackage{attachfile}
\usepackage{stmaryrd}
\usepackage{footnote}
\makesavenoteenv{tabular}
\usepackage{hyperref}
\usepackage{cleveref}

\makeatletter
\newcommand\footnoteref[1]{\protected@xdef\@thefnmark{\ref{#1}}\@footnotemark}
\makeatother

\hypersetup{
	pdfborder={0 0 0}
}

\newcommand{\imp}{\rightarrow}

\newcommand{\Pb}{\mathbb{P}}
\newcommand{\Qb}{\mathbb{Q}}

\newcommand{\Acal}{\mathcal{A}}
\newcommand{\Bcal}{\mathcal{B}}

\newcommand{\Fcal}{\mathcal{F}}

\newcommand{\Ucal}{\mathcal{U}}

\newcommand{\Ical}{\mathcal{I}}

\newcommand{\Lcal}{\mathcal{L}}
\newcommand{\Mcal}{\mathcal{M}}
\newcommand{\Pcal}{\mathcal{P}}

\newcommand{\Scal}{\mathcal{S}}

\newcommand{\uh}{{\upharpoonright}}

\renewcommand{\setminus}{\smallsetminus}

\def\qt#1{``#1''}%

\newcommand{\s}[1]{\ensuremath{\sf{#1}}}

\DeclareMathOperator{\rca}{\s{RCA}_0}
\DeclareMathOperator{\piooca}{\Pi^1_1\s{CA}}
\DeclareMathOperator{\aca}{\s{ACA}}
\DeclareMathOperator{\atr}{\s{ATR}}
\DeclareMathOperator{\wkl}{\s{WKL}}

\DeclareMathOperator{\rt}{\s{RT}}
\DeclareMathOperator{\srt}{\s{SRT}}

\DeclareMathOperator{\coh}{\s{COH}}

\newcommand{\qvdash}{\operatorname{{?}{\vdash}}}
\newcommand{\nqvdash}{\operatorname{{?}{\nvdash}}}

\DeclareMathOperator{\dom}{dom}

\newtheoremstyle{custom}
  {10pt}
  {10pt}
  {\normalfont}
  {}
  {\bfseries}
  {}
  { }
  {}

\theoremstyle{custom}

\usepackage{xcolor}	
\usepackage{soul}

\newtheorem{theorem}{Theorem}[section]
\newtheorem{lemma}[theorem]{Lemma}

\theoremstyle{definition}
\newtheorem{definition}[theorem]{Definition}

\newtheorem*{statement}{Statement}

\theoremstyle{remark}
\newtheorem{remark}[theorem]{Remark}
\newtheorem {question}[theorem]{Question}

\numberwithin{equation}{section}

\newtheoremstyle{noparens}%
  {}{}%
{}{}%
{\bfseries}{.}%
{ }%
{\thmname{#1}\thmnumber{ #2}\thmnote{ #3}}
\theoremstyle{noparens}
\newtheorem*{theorem*}{Theorem}

\title{$\mathsf{SRT}^2_2$ does not imply $\mathsf{RT}^2_2$ in $\omega$-models}
\author{
  Benoit Monin \and Ludovic Patey
}

\begin{document}

\begin{abstract}
We complete a 40-year old program on the computability-theoretic analysis of Ramsey's theorem, starting with Jockusch in 1972, and improving a result of Chong, Slaman and Yang in 2014. Given a set $X$, let $[X]^n$ be the collection of all $n$-element subsets of $X$. Ramsey's theorem for $n$-tuples asserts the existence, for every finite coloring of $[\omega]^n$, of an infinite set $X \subseteq \omega$ such that $[X]^n$ is monochromatic. The meta-mathematical study of Ramsey has a rich history, with several long-standing open problems and seminal theorems, including Seetapun's theorem in 1995 and Liu's theorem in 2012 about Ramsey's theorem for pairs. The remaining question about the study of Ramsey's theorem from a computational viewpoint was the relation between Ramsey's theorem for pairs ($\mathsf{RT}^2_2$) and its restriction to stable colorings ($\mathsf{SRT}^2_2$), that is, colorings admitting a limit behavior. Chong, Slaman and Yang first proved that $\mathsf{SRT}^2_2$ does not formally imply $\mathsf{RT}^2_2$ in a proof-theoretic sense, using non-standard models of reverse mathematics. In this article, we answer the open question whether this non-implication also holds within the framework of computability theory. More precisely, we construct a $\omega$-model of $\mathsf{SRT}^2_2$ which is not a model of $\mathsf{RT}^2_2$. For this, we design a new notion of effective forcing refining Mathias forcing using the notion of largeness classes.
\end{abstract}

\maketitle

\section{Introduction}

In this article, we prove that the restriction of Ramsey's theorem for pairs to stable colorings is not equivalent to its full version over $\omega$-models\footnote{The authors are thankful to Damir Dzhafarov for insightful comments and discussions.}. This answers a major open question of modern  \emph{reverse mathematics}, asked by Cholak, Jockusch and Slaman~\cite{Cholak2001strength} and Chong, Slaman and Yang~\cite{Chong2014metamathematics} and completes the 40-years old program started by Jockusch in 1972 about the analysis of Ramsey's theorem from a computability-theoretic viewpoint. 

\subsection{Reverse mathematics and Ramsey's theorem}

Reverse mathematics is a foundational program started by Harvey Friedman in 1975, whose goal is to find the weakest axioms needed to prove ordinary theorems. It uses the framework of second-order arithmetics, with a base theory, $\rca$, capturing “computable mathematics". The early study of reverse mathematics revealed the existence of four linearly ordered big systems $\wkl$, $\aca$, $\atr$, and $\piooca$ (in increasing order), such that, given an ordinary theorem, it is very likely either to be provable in $\rca$, or provably equivalent to one of the four systems in $\rca$. These systems together with $\rca$ are known as the \qt{Big Five}. By its success in finding the exact axioms needed for the large majority of theorems and its foundational consequences, in particular its partial answer to Hilbert's program of finitistic reductionnism~\cite{Simpson1988Partial}, reverse mathematics is cited among the 100 key breakthroughs in mathematics~\cite{Elwes2013Math}. Among the Big Five, $\wkl$ stands for \qt{weak K\"onig's lemma}, and asserts that every infinite binary tree admits an infinite path, while $\aca$ is the comprehension axiom restricted to arithmetical formulas. $\wkl$ can be thought of as capturing compactness arguments, and $\aca$ is equivalent to the existence, for every set $X$, of the halting set relative to~$X$. See Simpson~\cite{Simpson2009Subsystems} for an introduction to reverse mathematics.

Among the theorems studied in reverse mathematics, Ramsey's theorem received a special attention from the community, since Ramsey's theorem for pairs historically was the first theorem known to escape the Big Five phenomenon. Given a set of integers $X$, $[X]^n$ denotes the set of all $n$-element subsets over $X$. For a coloring $f : [\omega]^n \to k$, a set of integers $H$ is \emph{homogeneous} if $f$ is constant over $[H]^n$.

\begin{statement}[Ramsey's theorem]
$\rt^n_k$: \qt{Every $k$-coloring of $[\omega]^n$ admits an infinite homogeneous set}.
\end{statement}

In particular, $\rt^1_k$ is the infinite pigeonhole principle for $k$-partitions. Ramsey's theorem and its consequences are notoriously hard to analyse from a computability-theoretic viewpoint. Jockusch~\cite{Jockusch1972Ramseys} proved that $\rt^n_k$ is equivalent to $\aca$ whenever $n \geq 3$, thereby showing that $\rt^n_k$ satisfies the Big Five phenomenon. The question of whether $\rt^2_k$ implies $\aca$ was a longstanding open question, until Seetapun~\cite{Seetapun1995strength} proved that $\rt^2_k$ is strictly weaker than $\aca$. Later, Jockusch~\cite{Jockusch1972Ramseys,Jockusch197201} and Liu~\cite{Liu2012RT22} showed that $\rt^2_k$ is incomparable with $\wkl$, and therefore that $\rt^2_k$ is not even linearly ordered with the Big Five. See Hirschfeldt~\cite{Hirschfeldt2015Slicing} for an introduction to the reverse mathematics of Ramsey's theorem.

\subsection{Stable Ramsey's theorem for pairs and cohesiveness}

In order to understand better the computational and proof-theoretic content of Ramsey's theorem for pairs, Cholak, Jockusch and Slaman~\cite{Cholak2001strength} decomposed it into two statements, namely, stable Ramsey's theorem for pairs, and cohesiveness.  A coloring of pairs $f : [\omega]^2 \to k$ is \emph{stable} if for every $x \in \omega$, $\lim_y f(\{x, y\})$ exists. An infinite set $C$ is \emph{cohesive} for a countable sequence of sets $R_0, R_1, \dots$ if $C \subseteq^{*} R_i$ or $C \subseteq^{*} \overline{R}_i$ for every $i \in \omega$, where $\subseteq^*$ means inclusion but for finitely many elements.

\begin{statement}[Stable Ramsey's theorem for pairs]
$\srt^2_k$: \qt{Every stable $k$-coloring of $[\omega]^2$ admits an infinite homogeneous set}.
\end{statement}

\begin{statement}[Cohesiveness]
$\coh$: \qt{Every countable sequence of sets has a cohesive set}.
\end{statement}

Cholak, Jockusch and Slaman~\cite{Cholak2001strength} and Mileti~\cite{Mileti2004Partition} proved the equivalence over $\rca$ between $\rt^2_k$ and $\srt^2_k \wedge \coh$. They naturally wondered whether this decomposition is non-trivial, in the sense that both statements $\srt^2_k$ and $\coh$ are strictly weaker than $\rt^2_k$. Hirschfeldt, Jockusch, Kjoss-Hanssen, Lempp and Slaman~\cite{Hirschfeldt2008strength} partially answered the question by proving that  $\coh$ is strictly weaker than $\rt^2_2$ over $\rca$. The question of whether $\srt^2_2$ implies $\rt^2_2$ over $\rca$ remained a long-standing open question. Since $\rt^2_2$ is equivalent to $\srt^2_2 \wedge \coh$, this is equivalent to the question of whether $\srt^2_2$ implies $\coh$ over $\rca$.

From a computability-theoretic viewpoint, stable Ramsey's theorem for pairs and two $k$ colors is equivalent to combinatorially simpler statement called $\mathsf{D}^2_k$.

\begin{statement}
$\mathsf{D}^n_k$: \qt{For every $\Delta^0_n$ $k$-partition of $\omega$, there is an infinite subset of one of the parts}.
\end{statement}

Chong, Lempp and Yang~\cite{Chong2010role}, proved that the computable equivalence between $\rt^2_k$ and $\mathsf{D}^2_k$ also holds over $\rca$.
The cohesiveness principle also admits a nice computability-theoretic characterization. Jockusch and Stephan~\cite{Jockusch1993cohesive} proved that the sequence of all primitive recursive sets is a maximally difficult computable instance of $\coh$. The cohesive sets for this sequence are called \emph{p-cohesive} and their Turing degrees are precisely the ones whose jump is PA over $\emptyset'$, that is, the degrees whose jump can compute a path through any $\Delta^0_2$ infinite binary tree. The following computability-theoretic question is therefore closely related to the previous question.

\begin{question}
Does every $\Delta^0_2$ set have an infinite subset in it or its complement whose jump is not of PA degree over $\emptyset'$?
\end{question}

One natural approach to separate $\srt^2_2$ from $\rt^2_2$ would be to prove that every $\Delta^0_2$ set admits an infinite subset $G$ in it or its complement of low degree, that is, $G' \leq_T \emptyset'$. However, Downey, Hirschfeldt, Lempp and Solomon~\cite{Downey200102} constructed a $\Delta^0_2$ set with no low infinite subset of it or its complement. Very surprisingly, Chong, Slaman and Yang~\cite{Chong2014metamathematics} answered the $\srt^2_2$ vs $\rt^2_2$ question by constructing a model of $\rca + \srt^2_2$ with only low sets, which is not a model of $\rt^2_2$. The solution to this apparent paradox was the use of a non-standard model of $\rca$ in which $\Sigma^0_2$ induction fails. The sets of this model are low \emph{within the model}, but not low in the meta-theory. The construction of Downey, Hirschfeldt, Lempp and Solomon~\cite{Downey200102} requires $\Sigma^0_2$ induction to be carried out.

Although the proof of Chong, Slaman and Yang~\cite{Chong2014metamathematics} formally separated $\srt^2_2$ from $\rt^2_2$ over $\rca$, the separation was not fully satisfactory, for two reasons. First, it leaves open the question of whether $(\forall k)\srt^2_k$ implies $\rt^2_2$ which as also asked by Cholak, Jockusch and Slaman~\cite{Cholak2001strength}. Indeed, $(\forall k)\srt^2_k$ implies $\Sigma^0_2$ induction, and therefore cannot have any models with only low sets. The second reasons is that the separations of Chong, Slaman and Yang~\cite{Chong2014metamathematics} exploits the failure of a property which is known to hold in standard models. A structure in second-order arithmetics is a tuple $(M, \Scal, 0, 1, +, \times, <)$ where $M$ denotes the set of integers, together with some constants $0$ and $1$, some binary operations $+$ and $\times$ and an order relation $<$. $\Scal$ is a collection of subsets of $M$ representing the second-order part. Among these structures, we are particularly interested in those whose first-order part consists of the standard integers, together with their natural operations. These structures are called \emph{$\omega$-structures}, and are fully specified by their second-order part~$\Scal$. Chong, Slaman and Yang~\cite{Chong2014metamathematics} naturally asked the following question:

\begin{question}
Is every $\omega$-model of $\rca \wedge \srt^2_2$ a model of $\rt^2_2$?
\end{question}

This question had an important impact in the development of reverse mathematics, and computability theory in general, not only by its self interest, but also by range of related questions, new techniques and intellectual emulation it generated in the community. Several articles are dedicated to this question~\cite{Cholak2019Some,Chong2010role,Dorais2016uniform,Dzhafarov2014Cohesive,Dzhafarov2019COH,Dzhafarov2016Strong,Hirschfeldt2016notions,Monin2018Pigeons,Patey2016weakness,Patey2017Controlling} and led to the rediscovery of Weihrauch degrees by Dorais, Dzhafarov, Hirst, Mileti and Shafer~\cite{Dorais2016uniform}, and the design of the computable reduction by Dzhafarov~\cite{Dzhafarov2014Cohesive}. Dzhafarov~\cite{Dzhafarov2014Cohesive,Dzhafarov2016Strong} obtained partial separations by proving that $\coh$ is neither Weihrauch reducible, nor strongly computably reducible to~$\srt^2_2$. The most recent improvement is a proof by Dzhafarov and Patey~\cite{Dzhafarov2019COH} proving that $\coh$ is not Weihrauch reducible to $\srt^2_2$ even when finitely many Turing functionals are allowed.

In this article, we answer the question by separating $\rca \wedge \srt^2_2$ from $\rt^2_2$ on $\omega$-models. For this, we prove the following main theorem.

\begin{theorem}\label{thm:jump-pa-avpoidance-d22}
For every set $Z$ whose jump is not of PA degree over $\emptyset'$ and every $\Delta^{0,Z}_2$ set $A$, there is an infinite subset $G \subseteq A$ or $G \subseteq \overline{A}$ such that $(G \oplus Z)'$ is not of PA degree over $\emptyset'$.
\end{theorem}

This theorem is based on a second-jump control initially developed by Cholak, Jockusch and Slaman~\cite{Cholak2001strength}, and then successively refined by Wang~\cite{Wang2014Cohesive}, Patey~\cite{Patey2017Controlling} and Monin and Patey~\cite{Monin2018Pigeons}. The techniques are combined with combinatorial ideas of Liu~\cite{Liu2012RT22,Liu2015Cone}.
\Cref{thm:jump-pa-avpoidance-d22} can be iterated to construct an $\omega$-model of $\rca \wedge \srt^2_2$ containing no set whose jump is of PA degree over $\emptyset'$, from which we deduce the following theorem.

\begin{theorem}\label{thm:srt22-coh-omega-models}
There is an $\omega$-model of $\rca \wedge \srt^2_2$ which is not a model of $\rt^2_2$.
\end{theorem}

This answers a question of Chong, Slaman and Yang~\cite{Chong2014metamathematics}, but also of Cholak, Jockusch and Slaman~\cite{Cholak2001strength} since any $\omega$-model of $\srt^2_2$ is a model of $(\forall k)\srt^2_k$.

By being an adaptation and generalization of the first and second-jump control of Cholak, Jockusch and Slaman, the nature of the proof of \Cref{thm:jump-pa-avpoidance-d22} supports the idea that this framework is the appropriate one for the computable and proof-theoretic analysis of Ramsey-like theorems. This gives a reasonable hope to prove general decidability theorems on Ramsey's theorem, in the spirit of~\cite{Patey2019Ramsey}

\begin{figure}[htbp]
\begin{center}
\begin{tikzpicture}[x=2.5cm, y=1.3cm,
	node/.style={minimum size=2em},
	impl/.style={draw,very thick,-latex},
	strict/.style={draw, thick, -latex, double distance=2pt},
	nonimpl/.style={draw, very thick, dotted, -latex}]

	\node[node] (piooca) at (2, 6) {$\piooca$};
	\node[node] (atr) at (2, 5) {$\atr$};
	\node[node] (aca) at (2, 4) {$\aca$};
	\node[node] (wkl) at (2, 3) {$\wkl$};
	\node[node] (rca) at (2, 1) {$\rca$};
	\node[node] (rtn2) at (1, 4) {$\rt^n_2$};
	\node[node] (rt22) at (1, 3) {$\rt^2_2$};
	\node[node] (srt22) at (0.5, 2) {$\srt^2_2$};
	\node[node] (coh) at (1.5, 2) {$\coh$};

	\draw[impl] (piooca) -- (atr);
	\draw[impl] (atr) -- (aca);
	\draw[impl] (aca) -- (wkl);
	\draw[impl] (wkl) -- (rca);
	\draw[impl,<->] (aca) -- (rtn2);
	\draw[impl] (rtn2) -- (rt22);
	\draw[impl] (rt22) -- (srt22);
	\draw[impl] (rt22) -- (coh);
	\draw[impl] (coh) -- (rca);
	\draw[impl] (srt22) -- (rca);
\end{tikzpicture}

\end{center}
\caption{Summary diagram of implications between statements over $\rca$, and over $\omega$-models. All the implications are strict, and the missing implications are separations.}
\end{figure}
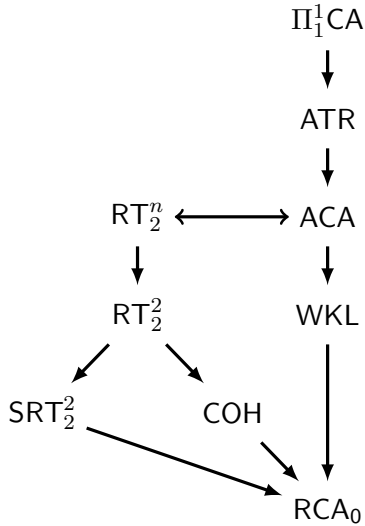

\subsection{Definitions and notation}

A \emph{binary string} is an ordered tuple of bits $a_0, \dots, a_{n-1} \in \{0, 1\}$.
The empty string is written $\epsilon$. A \emph{binary sequence} (or a \emph{real}) is an infinite listing of bits $a_0, a_1, \dots$.
Given $s \in \omega$,
$2^s$ is the set of binary strings of length $s$ and
$2^{<s}$ is the set of binary strings of length $<s$. As well,
$2^{<\omega}$ is the set of binary strings
and $2^{\omega}$ is the set of binary sequences.
Given a string $\sigma \in 2^{<\omega}$, we use $|\sigma|$ to denote its length.
Given two strings $\sigma, \tau \in 2^{<\omega}$, $\sigma$ is a \emph{prefix}
of $\tau$ (written $\sigma \preceq \tau$) if there exists a string $\rho \in 2^{<\omega}$
such that $\sigma \rho = \tau$. Given a sequence $X$, we write $\sigma \prec X$ if
$\sigma = X \uh n$ for some $n \in \omega$.
A binary string $\sigma$ can be interpreted as a finite set $F_\sigma = \{ x < |\sigma| : \sigma(x) = 1 \}$. We write $\sigma \subseteq \tau$ for $F_\sigma \subseteq F_\tau$.
We write $\#\sigma$ for the size of $F_\sigma$.

A \emph{binary tree} is a set of binary strings $T \subseteq 2^{<\omega}$ which is closed downward under the prefix relation. A \emph{path} through $T$ is a binary sequence $P \in 2^\omega$ such that every initial segment belongs to $T$.

A \emph{Turing ideal} $\Ical$ is a non-empty collection of sets which is closed downward under the Turing reduction and closed under the effective join, that is, $(\forall X \in \Ical)(\forall Y \leq_T X) Y \in \Ical$ and $(\forall X, Y \in \Ical) X \oplus Y \in \Ical$, where $X \oplus Y = \{ 2n : n \in X \} \cup \{ 2n+1 : n \in Y \}$. A \emph{Scott set} is a Turing ideal $\Ical$ such that every infinite binary tree $T \in \Ical$ has a path in $\Ical$. In other words, a Scott set is the second-order part of an $\omega$-model of $\rca + \wkl$.
A countable Turing ideal $\Mcal$ is \emph{coded} by a set $X$
if $\Mcal = \{ X_n : n \in \omega \}$ with $X = \bigoplus_n X_n$.
A formula is $\Sigma^0_1(\Mcal)$ (resp.\ $\Pi^0_1(\Mcal)$) if it is $\Sigma^0_1(X)$ (resp.\ $\Pi^0_1(X)$) for some $X \in \Mcal$.

Given two sets $A$ and $B$, we denote by $A < B$ the formula
$(\forall x \in A)(\forall y \in B)[x < y]$.
We write $A \subseteq^{*} B$ to mean that $A - B$ is finite, that is,
$(\exists n)(\forall a \in A)(a \not \in B \imp a < n)$.
A \emph{$k$-cover} of a set $X$ is a sequence of sets $Y_0, \dots, Y_{k-1}$ such that $X \subseteq Y_0 \cup \dots \cup Y_{k-1}$.

\section{Background and sketch of the proof}\label{sect:outline}

In this section, we give a sketch of the proof that every $\Delta^0_2$ set $A$ admits an infinite subset in it or its complement, whose jump is not of PA degree relative to $\emptyset'$. Many claims are formally proven in their full generality in \Cref{sect:delta2-jump-not-pa}. The proof is done by a variant of Mathias forcing with an effective second-jump control, that is, a notion of forcing whose forcing relation for $\Sigma^0_2$ and $\Pi^0_2$ formulas is $\Sigma^0_2$ and $\Pi^0_2$, respectively. In the remainder of this section, fix a $\Delta^0_2$ set $A$ and let $A^0 = A$ and $A^1 = \overline{A}$.

\subsection{First-jump control}

Cholak, Jockusch and Slaman~\cite{Cholak2001strength} designed a notion of forcing for constructing subsets of $A^0$ or $A^1$, with a good first-jump control. This notion of forcing is a variant of Mathias forcing whose conditions are tuples $(\sigma^0, \sigma^1, X)$ where $\sigma^0 \subseteq A^0$ and $\sigma^1 \subseteq A^1$ are finite strings representing the stem of the two sets $G^0 \subseteq A^0$ and $G^1 \subseteq A^1$ that we are building. The set $X \subseteq \omega$ is an infinite set belonging to some fixed Scott set $\Mcal$, and serves as a reservoir of elements to add to the stems $\sigma^0$ and $\sigma^1$. For example, by the low basis theorem, $X$ can be chosen to be of low degree. We furthermore require that $\min X > \max( \sigma^0, \sigma^1)$.  According to the intuition, a condition $d = (\tau^0, \tau^1, Y)$ \emph{extends} a condition $c = (\sigma^0, \sigma^1, X)$ (written $d \leq c$) if the reservoir gets more restrictive, that is, $Y \subseteq X$, and if the stems are extended only with elements coming from the reservoir $X$, that is, $\sigma^i \preceq \tau^i$ and $\tau^i - \sigma^i \subseteq X \cap A^i$. The combinatorics of CJS provide a way to decide $\Sigma^0_1$ formulas without referring to the set $A$ which is computationally too complex for the question.

\begin{definition}
Let $c = (\sigma^0, \sigma^1, X)$ be a condition, $i < 2$ and $\Phi_e(G, x)$ be a $\Delta_0$ formula.
\begin{itemize}
	\item[(a)] $c \Vdash^i (\exists x)\Phi_e(G,x)$ if there is some $x \in \omega$ such that $\Phi_e(\sigma^i, x)$ holds.
	\item[(b)] $c \Vdash^i (\forall x)\neg \Phi_e(G,x)$ if for every $x \in \omega$ and every $\rho \subseteq X$, $\neg \Phi_e(\sigma^i \cup \rho, x)$ holds.
\end{itemize}
\end{definition}

Note that the set $A$ does not appear in the definition of the forcing relation for $\Pi^0_1$ formulas. The forcing relation for $\Sigma^0_1$ and $\Pi^0_1$ formulas is therefore $\Sigma^0_1(X)$ and $\Pi^0_1(X)$, respectively, where $X$ is the reservoir of the condition. There is no reason to consider that either a $\Sigma^0_1$ formula or its negation can be forced on each side of a condition. However, the following forcing question which is at the heart of the CJS combinatorics ensures that this can be achieved on at least one side.

\begin{definition}
Given a condition $c = (\sigma^0, \sigma^1, X)$ and two $\Delta_0$ formulas $\Phi_{e_0}(G, x)$ and $\Phi_{e_1}(G, x)$, define $c \qvdash (\exists x)\Phi_{e_0}(G^0, x) \vee (\exists x)\Phi_{e_1}(G^1, x)$ to hold if for every 2-cover $Z^0 \cup Z^1 = X$, there is some side $i < 2$, some $x \in \omega$ and some finite set $\rho \subseteq Z^i$ such that $\Phi_{e_i}(\sigma^i \cup \rho, x)$ holds.
\end{definition}

The forcing question for $\Sigma^0_1$ formulas is also $\Sigma^0_1(X)$, and satisfies the following property.

\begin{lemma}[Cholak, Jockusch and Slaman~\cite{Cholak2001strength}]\label{lem:forcing-question-first-jump-spec}
Let $c$ be a condition, and $\Phi_{e_0}(G, x)$ and $\Phi_{e_1}(G, x)$ be two $\Delta_0$ formulas.
\begin{itemize}
	\item[(a)] If $c \qvdash (\exists x)\Phi_{e_0}(G^0, x) \vee (\exists x)\Phi_{e_1}(G^1, x)$, then there is some $d \leq c$ and some $i < 2$ such that $d \Vdash^i (\exists x)\Phi_{e_i}(G, x)$.
	\item[(b)] If $c \nqvdash (\exists x)\Phi_{e_0}(G^0) \vee (\exists x)\Phi_{e_1}(G^1)$, then there is some $d \leq c$ and some $i < 2$ such that $d \Vdash^i (\forall x)\neg \Phi_{e_i}(G, x)$.
\end{itemize}
\end{lemma}
\begin{proof}
Suppose $c \qvdash (\exists x)\Phi_{e_0}(G^0, x) \vee (\exists x)\Phi_{e_1}(G^1, x)$ holds. Then letting $Z^0 = X \cap A^0$ and $Z^1 = X \cap A^1$, there is some side $i < 2$, some $x \in \omega$ and some finite set $\rho \subseteq X \cap A^i$ such that $\Phi_{e_i}(\sigma^i \cup \rho)$ holds. The condition $d = (\sigma^i \cup \rho, \sigma^{1-i}, X \cap (\max \rho, \infty))$ is an extension of $c$ such that $d \Vdash^i (\exists x)\Phi_{e_i}(G, x)$.

Suppose now that $c \nqvdash (\exists x)\Phi_{e_0}(G^0, x) \vee (\exists x)\Phi_{e_1}(G^1, x)$. Let $\Pcal$ be the collection of all the sets $Z^0 \oplus Z^1$ with $Z^0 \cup Z^1 = X$ such that for every $i < 2$, every $x \in \omega$ and every finite set $\rho \subseteq Z^i$, $\Phi_{e_i}(\sigma^i \cup \rho, x)$ does not hold. By assumption, $\Pcal$ is a non-empty $\Pi^{0,X}_1$ class, so since $X$ belongs to the Scott set $\Mcal$, there is some 2-cover $Z^0 \cup Z^1 = X$ such that $Z^0 \oplus Z^1 \in \Pcal \cap \Mcal$. Let $i < 2$ be such that $Z^i$ is infinite. Then the condition $d = (\sigma^0, \sigma^1, Z^i)$ is an extension of $c$ such that $d \Vdash^i (\forall x)\neg \Phi_{e_i}(G, x)$.
\end{proof}

By a pairing argument (if for every pair $m, n \in \omega$, $m \in A$ or $n \in B$, then $A = \omega$ or $B = \omega$), if a filter $\Fcal$ is sufficiently generic, there is some side $i$
such that for every $\Sigma^0_1$ formula $\varphi(G)$, there is some $c \in \Fcal$ such that $c \Vdash^i \varphi(G)$ or $c \Vdash^i \neg \varphi(G)$.
Note that in the proof of \Cref{lem:forcing-question-first-jump-spec}, the new reservoir refining $X$ is either $X$ truncated by finitely many elements, or in the form $X \cap Z^0$ or $X \cap Z^1$ for some 2-cover $Z^0 \cup Z^1 = \omega$ such that $Z^0 \oplus Z^1 \in \Mcal$.

\subsection{Second-jump control}

The forcing relation for $\Sigma^0_2$ formulas can be defined inductively by stating that $c \Vdash^i (\exists x)(\forall y)\Phi_e\allowbreak(G, x, y)$ if $c \Vdash^i (\forall y)\Phi_e(G, x, y)$ for some $x \in \omega$. The relation $c \Vdash^i (\forall y)\Phi_e(G, x, y)$ is $\Pi^0_1(X)$ where $X$ is the reservoir of $c$. In particular, whenever $X$ is low, then the relation  $c \Vdash^i (\exists x)(\forall y)\Phi_e(G, x, y)$ is $\Sigma^0_2$.

The definition of a forcing relation for $\Pi^0_2$ formulas is more problematic. A $\Pi^0_2$ formula $(\forall x)(\exists y)\neg \Phi_e(G, x, y)$ can bee seen as a collection of $\Sigma^0_1$ properties $\langle (\exists y)\neg \Phi_e(G, x, y) : x \in \omega \rangle$ which all need to be forced. It is usually impossible to force infinitely many $\Sigma^0_1$ properties simultaneously, and one has to satisfy them one by one, by proving that for every $x \in \omega$, the set of collections forcing $(\exists y)\neg \Phi_e(G, x, y)$ is dense below $c$. The forcing relation $c \Vdash^i (\forall x)(\exists y)\neg \Phi_e(G, x, y)$ is therefore naturally defined by the statement
	$$
	(\forall x \in \omega)(\forall c_1 \leq c)(\exists c_2 \leq c_1) c_2 \Vdash^i (\exists y)\neg \Phi_e(G, x, y)
	$$
	This definition of the forcing relation for $\Pi^0_2$ formulas is however too complex from a computational viewpoint. The main complexity comes from the description of the reservoir refinement, saying \qt{there exists an infinite set $Y \in \Mcal$ such that $Y \subseteq X$.}

Cholak, Jockusch and Slaman~\cite{Cholak2001strength} went around this problem by observing that the only operations needed on the reservoirs to provide a good first-jump control were truncation and splitting. In some sense, the notion of forcing should not be considered as a variant of Mathias forcing since the refinement operation of the reservoirs does not need to be the arbitrary subset relation. The relation $c = (\sigma^0, \sigma^1, X) \Vdash^i (\forall x)(\exists y)\neg \Phi_e(G, x, y)$ is then translated into the sentence \qt{For every $x \in \omega$, every $\tau^0$ and $\tau^1$ extending $\sigma^0$ and $\sigma^1$ with elements from $X \cap A^0$ and $X \cap A^1$, respectively, the collection of reservoirs $Y$ such that $(\tau^0, \tau^1, Y) \not\Vdash (\forall y)\Phi_e(G, x, y)$ is \emph{large}}, for some notion of largeness which enables truncation and splitting. We now give a modern presentation of the notion of forcing designed by Cholak, Jockusch and Slaman with a good second-jump control.

\begin{definition}
A class $\Acal \subseteq 2^\omega$ is a \emph{largeness class} if it satisfies the following properties:
\begin{itemize}
	\item[(1)] For every $X \in \Acal$ and $Y \supseteq X$, $Y \in \Acal$
	\item[(2)] For every $k \in \omega$ and every $X_0 \cup \dots \cup X_{k-1} \supseteq \omega$, there is some $j < k$ such that $X_j \in \Acal$
\end{itemize}
\end{definition}

Fix a countable Scott set $\Mcal = \{X_0, X_1, \dots \}$ coded by a set $M$ of low degree, that is, $M = \bigoplus_i X_i$ and $M' \leq_T \emptyset'$. Such a Scott set exists by the Low Basis theorem~\cite{Jockusch197201}.
Fix a uniformly $M$-computable enumeration $\Ucal_0, \Ucal_1, \dots$ of all the upward-closed $\Sigma^{0,X}_1$ sub-classes of $2^\omega$ for every $X \in \Mcal$. We are particularly interested in largeness classes in the form $\Ucal_C = \bigcap_{e \in C} \Ucal_e$ for some $\Delta^0_2$ set of indices $C \subseteq \omega$.
Given an infinite set $X \in \Mcal$, we let $\Lcal_X$ be the largeness class of all $Z \subseteq \omega$ such that $Z \cap X$ is infinite. 

Let us enrich the previous notion of forcing with a $\Delta^0_2$ set of indices $C \subseteq \omega$ representing a largeness class of the form $\Ucal_C$ and to which the reservoir must belong. A forcing condition is therefore a tuple $(\sigma^0, \sigma^1, X, C)$ where $\sigma^0 \subseteq A^0$, $\sigma^1 \subseteq A^1$, $X$ is an infinite set belonging to $\Mcal$ (in particular of low degree), $C \subseteq \omega$ is a $\Delta^0_2$ set such that $\Ucal_C$ is a largeness class, and $\Ucal_C \subseteq \Lcal_X$.

\begin{remark}
Instead of asking $X \in \Ucal_C$, we require the stronger fact that $\Ucal_C \subseteq \Lcal_X$. Since $X$ is of low degree, $\Lcal_X$ can be put in the form $\Ucal_C$ for some $\Delta^0_2$ set $C \subseteq \omega$. The requirement $X \in \Ucal_C$ is not strong enough, as there might be some cover $Y_0 \cup \dots Y_{k-1} \supseteq X$ such that $Y_j \not \in \Ucal_C$ for every $j < k$. There might also be some $D \supseteq C$ such that $\Ucal_D$ is a largeness sub-class of~$\Ucal_C$, but $X \not \in \Ucal_D$. By requiring that $\Ucal_C \subseteq \Lcal_X$, we ensure that every largeness subclass $\Ucal_D \subseteq \Ucal_C$ is a largeness class \emph{within $X$}, that is, for every $k$ and every $k$-cover $Y_0 \cup\dots \cup Y_{k-1} \supseteq X$, there is some $j < k$ such that $Y_j \in \Ucal_D$. In particular $X \in \Ucal_D$.
\end{remark}

A condition $(\tau^0, \tau^1, Y, D)$ extends $(\sigma^0, \sigma^1, X, C)$ if $(\tau^0, \tau^1, Y)$ extends $(\sigma^0, \sigma^1, X)$ as before, except that $D \supseteq C$, which means that $\Ucal_D \subseteq \Ucal_C$.
The forcing relation for $\Sigma^0_1$ and $\Pi^0_1$ formulas is left unchanged. In particular, it does not depend on the largeness class $\Ucal_C$.

\begin{definition}
Given a $\Delta_0$ formula $\Phi_e(G, x, y)$, a finite set $\sigma \in 2^{<\omega}$ and some integer $x \in \omega$, let $\zeta(e, \sigma, x)$ be an index of the $\Sigma^0_1$ class
$$
\Ucal_{\zeta(e, \sigma, x)} = \{ X : (\exists \rho \subseteq X)(\exists y \in \omega) \neg \Phi_e(\sigma \cup \rho, x, y) \}
$$
\end{definition}

In other words, $\Ucal_{\zeta(e, \sigma, x)}$ is the collection of all reservoirs $X$ such that the Mathias condition $(\sigma, X)$ does not force $(\forall y)\Phi_e(G, x, y)$. We can now define a forcing relation for $\Sigma^0_2$ and $\Pi^0_2$ formulas whose complexities are $\Sigma^0_2$ and $\Pi^0_2$, respectively.

\begin{definition}
Let $c = (\sigma^0, \sigma^1, X, C)$ be a condition, $i < 2$ and $\Phi_e(G, x, y)$ be a $\Delta_0$ formula.
\begin{itemize}
	\item[(a)] $c \Vdash^i (\exists x)(\forall y)\Phi_e(G,x, y)$ if there is some $x \in \omega$ such that $c \Vdash^i (\forall y)\Phi_e(G, x, y)$ holds.
	\item[(b)] $c \Vdash^i (\forall x)(\exists y)\neg \Phi_e(G,x, y)$ if for every $x \in \omega$ and every $\rho \subseteq X \cap A^i$, $\zeta(e, \sigma^i \cup \rho, x) \in C$.
\end{itemize}
\end{definition}

The interpretation of the forcing relation for $\Sigma^0_2$ formulas is immediate. The forcing relation for a $\Pi^0_2$ formula $(\forall x)(\exists y)\neg \Phi_e(G, x, y)$ ensures that for every $x \in \omega$ and every extension $d \leq c$, $d \not \Vdash^i (\forall y)\Phi_e(G, x, y)$. Indeed, if $d = (\tau^0, \tau^1, Y, D)$, then for every $x \in \omega$, $Y \in \Ucal_D \subseteq \Ucal_C \subseteq \Ucal_{\zeta(e, \tau^i, x)}$. 

\begin{remark}
Note that if a condition $c$ forces a $\Sigma^0_2$ formula on a side $i < 2$, then the formula will hold on $G^i = \bigcup \{ \sigma^i : (\sigma^0, \sigma^1, X, C) \in \Fcal \}$ for every filter $\Fcal$ containing $c$.
On the other hand, if $c$ forces a $\Pi^0_2$ formula on side $i$,
then the filter $\Fcal$ must be sufficiently generic for the property to hold on $G^i$. More precisely, the forcing relation for a formula $(\forall x)(\exists y)\neg \Phi_e(G, x, y)$ states that for every $x \in \omega$, the formula $(\forall y)\Phi_e(G, x, y)$ will never be forced. One can deduce that $(\exists y)\neg \Phi_e(G^i, x, y)$ whenever the side $i$ is 1-generic, meaning that every $\Sigma^0_1$ formula or its negation is forced on side~$i$. However, the disjunctive nature of the first-jump control guarantees only that at least one of the sides will be 1-generic. Therefore, one can ensure only on one side that the forced $\Pi^0_2$ formulas will actually hold.
\end{remark}

Because of the assumption that $A$ is $\Delta^0_2$, we can design a non-disjunctive forcing question for $\Sigma^0_2$ formulas, which will be $\Sigma^0_2$. This enables us to prove that for \emph{each} side, the set of conditions forcing a $\Sigma^0_2$ formula or its negation is dense. However, by the previous remark, the forced properties are only guaranteed to hold on one side.

\begin{definition}
Let $c = (\sigma^0, \sigma^1, X, C)$ be a condition, $i < 2$ and $\Phi_e(G, x, y)$ be a $\Delta_0$ formula.
Let $c \qvdash^i (\exists x)(\forall y)\Phi_e(G, x, y)$ hold if
$$
\Ucal_C \cap \bigcap \{ \Ucal_{\zeta(e, \sigma^i \cup \rho, x)} : x \in \omega, \rho \subseteq X \cap A^i \}
$$
is not a largeness class.
\end{definition}

As we will see in \Cref{lem:decreasing-largeness-yields-largeness}, the forcing question holds if and only if there is a finite set $F \subseteq X$ and some $n \in \omega$ such that the class
$\Ucal_F \cap \bigcap \{ \Ucal_{\zeta(e, \sigma^i \cup \rho, x)} : x < n, \rho \subseteq X \cap A^i \uh n \}$ is not a largeness class. Note that the class is $\Sigma^{0,Z}_1$ for some $Z \in \Mcal$. A complexity analysis for the forcing question shows that not being a largeness class for a $\Sigma^{0,Z}_1$ class is $\Sigma^{0,Z}_2$ (see \Cref{lem:largeness-class-complexity}), hence $\Sigma^0_2$ whenever $Z$ is low. The forcing relation enjoys the following lemma, which shows in particular that every $\Sigma^0_2$ formula or its complement can be forced in \emph{each} side.

\begin{lemma}\label{lem:question-second-jump-simple}
Let $c$ be a condition, $i < 2$ and $\Phi_e(G, x, y)$ be a $\Delta_0$ formula.
\begin{itemize}
	\item[(1)] If $c \qvdash^i (\exists x)(\forall y)\Phi_e(G, x, y)$, then there is a $d \leq c$ such that $d \Vdash^i (\exists x)(\forall y)\Phi_e(G, x, y)$.
	\item[(2)] If $c \nqvdash^i (\exists x)(\forall y)\Phi_e(G, x, y)$, then there is a $d \leq c$ such that $d \Vdash^i (\forall x)(\exists y)\neg\Phi_e(G, x, y)$.
\end{itemize}
\end{lemma}
\begin{proof}
Say $c = (\sigma^0, \sigma^1, X, C)$.

Case 1: $c \qvdash^i (\exists x)(\forall y)\Phi_e(G, x, y)$. Then there is a finite set $F \subseteq C$ and some $n \in \omega$ such that the class $\Ucal_F \cap \bigcap \{ \Ucal_{\zeta(e, \sigma^i \cup \rho, x)} : x < n, \rho \subseteq X \cap A^i \uh n \}$ is not a largeness class. Since the class is $\Sigma^{0,Z}_1$ for some $Z$ belonging to the Scott set $\Mcal$, there is a $k$-cover $Y_0 \cup \dots \cup Y_{k-1} = \omega$ belonging to $\Mcal$ such that for every $j < k$, $Y_j \not\in \Ucal_F \cap \bigcap \{ \Ucal_{\zeta(e, \sigma^i \cup \rho, x)} : x < n, \rho \subseteq X \cap A^i \uh n \}$. Let $j < k$ be such that $\Ucal_C \cap \Lcal_{X \cap Y_j}$ is a largeness class (see \Cref{lem:intersection-compatibility}). In particular, there is some $x \in \omega$ and some $\rho \subseteq X \cap A^i$  such that $Y_j \not \in \Ucal_{\zeta(e, \sigma^i \cup \rho, x)}$, hence $(\sigma^i \cup \rho, x) \Vdash (\forall y)\Phi_e(G, x, y)$. Let $D \supseteq C$ be a $\Delta^0_2$ set such that $\Ucal_D = \Ucal_C \cap \Lcal_{X \cap Y_j}$. The condition $(\sigma^i \cup \rho, \sigma^{1-i}, X \cap Y_j - \{0, \dots, \max \rho\}, D)$ is the desired extension.

Case 2: $c \nqvdash^i (\exists x)(\forall y)\Phi_e(G, x, y)$. Then let $D = C \cup \bigcup \{ \zeta(e, \sigma^i \cup \rho, x) : x \in \omega, \rho \subseteq X \cap A^i \}$. The condition $(\sigma^0, \sigma^1, X, D)$ is the desired extension.
\end{proof}

One can use the combinatorics of largeness classes to provide a more direct proof that all the forced $\Pi^0_2$ formulas must hold on one of the sides for every sufficiently generic filter. We say that a condition $c = (\sigma^0, \sigma^1, X, C)$ is \emph{$i$-valid} for some $i < 2$ if $X \cap A^i \in \Ucal_C$. Since either $X \cap A^0$ or $X \cap A^1$ belongs to $\Ucal_C$, every condition is $i$-valid for at least one $i < 2$. Moreover, $i$-validity is upward-closed under the extension relation, that is, if a condition $d$ is $i$-valid and $d \leq c$, then $c$ is $i$-valid as well. Therefore, for every filter $\Fcal$, there is at least one side $i < 2$ such that every condition is $i$-valid. When a condition $c$ is $i$-valid, one can make some progress on $\Sigma^0_1$ formulas on the $i$-th side, as states the following lemma.

\begin{lemma}
For every $i$-valid condition $c = (\sigma^0, \sigma^1, X, C)$
and every $\zeta(e, \sigma^i, x) \in C$, there is an extension $d \leq c$ such that $d \Vdash^i (\exists y)\neg \Phi_e(G, x, y)$.
\end{lemma}
\begin{proof}
Since $c$ is $i$-valid, then $X \cap A^i \in \Ucal_C \subseteq \Ucal_{\zeta(e, \sigma^i, x)}$. Thus there is some $y \in \omega$ and some $\rho \subseteq X \cap A^i$ such that $\neg \Phi_e(\sigma^i \cup \rho, x, y)$ holds. The condition $d = (\sigma^i \cup \rho, \sigma^{1-i}, X - \{0, \dots, \max \rho\}, C)$ is the desired extension.
\end{proof}

\subsection{Forcing a jump of non-PA degree over $\emptyset'$}

Now the main combinatorics of the second-jump control have been introduced, let us explain the core argument of forcing the jump of a solution not to be of PA degree over $\emptyset'$. A degree is PA over $\emptyset'$ if and only if
it computes a $\{0,1\}$-valued completion of $n \mapsto \Phi^{\emptyset'}_n(n)$.

Suppose we want to prove that for every $\Delta^0_2$ set $A$, there is an infinite set $G \subseteq A$ or $G \subseteq \overline{A}$ such that $G'$ does not compute a $\{0,1\}$-valued completion of $n \mapsto \Phi^{\emptyset'}_n(n)$. We need the following notion of valuation.

\begin{definition}
A \emph{valuation} is a partial function $p :\subseteq \omega \to 2$.
A valuation is \emph{$\emptyset'$-correct} if $p(n) = \Phi^{\emptyset'}_n(n)$ for all $n \in \dom(p)$. Two valuations $p, q$ are \emph{incompatible} if there is an $n \in \dom(p) \cap \dom(q)$ such that $p(n) \neq q(n)$.
\end{definition}

Fix condition $c$, a side $i < 2$ and a Turing functional $\Gamma$. In order to force $\Gamma^{G'}$ not to be a completion of $n \mapsto \Phi^{\emptyset'}_n(n)$, it is sufficient to find an extension $d \leq c$ such that one of the following holds:
\begin{itemize}
	\item[(1)] $d \Vdash^i \Gamma^{G'} \not \subseteq p$ for some $\emptyset'$-correct valuation $p$
	\item[(2)] $d \Vdash^i \Gamma^{G'} \subseteq p_0$ and $d \Vdash^i \Gamma^{G'} \subseteq p_1$ for two incompatible valuations $p_0$ and $p_1$
\end{itemize}
where $\Gamma^{G'} \not \subseteq p$ is the $\Sigma^0_2$ formula $(\exists n \in \dom p) \Gamma^{G'}(n)\downarrow \neq p(n)$, and $\Gamma^{G'} \subseteq p$ is the $\Pi^0_2$ formula $(\forall n \in \dom p) \Gamma^{G'}(n)\uparrow \vee\ \Gamma^{G'}(n)\downarrow =  p(n)$. In particular, forcing the $\Sigma^0_2$ formula for a $\emptyset'$-correct valuation ensures that $\Gamma^{G'}(n)\downarrow \neq \Phi^{\emptyset'}_n(n)\downarrow$, while forcing the $\Pi^0_2$ formula for two incompatible valuations forces the partiality of $\Gamma^{G'}$. In both cases, we force $\Gamma^{G'}$ not to be a completion of $n \mapsto \Phi^{\emptyset'}_n(n)$.
The following lemma is adapted from a combinatorial lemma proven by Liu (Lemma 6.6 in \cite{Liu2012RT22}) and will be proven in \Cref{lem:combi-liu-valuation}.

\begin{lemma}[Liu~\cite{Liu2012RT22}]\label{lem:combi-liu-valuation-simplified}
Let $W$ be a $\emptyset'$-c.e.\  set of valuations.
Either $W$ contains a $\emptyset'$-correct valuation, or for every $k$,
there are $k$ pairwise incompatible valuations outside of $W$.
\end{lemma}

Let us apply \Cref{lem:combi-liu-valuation-simplified} to the following $\emptyset'$-c.e.\ set of valuations
$$
W = \{ p : c \qvdash^i \Gamma^{G'} \not \subseteq p \}
$$

\noindent\textbf{Case 1}: $p \in W$ for some $\emptyset'$-correct valuation $p$. Then $c \qvdash^i \Gamma^{G'} \not \subseteq p$. By \Cref{lem:question-second-jump-simple}, there is an extension $d \leq c$ such that $d \Vdash^i \Gamma^{G'}(n)\not \subseteq p$.

\bigskip
\noindent\textbf{Case 2}: $W \cap \{p_0, p_1, p_2\} = \emptyset$ for three pairwise incompatible valuations $p_0, p_1$ and $p_2$. Say $c = (\sigma^0, \sigma^1, X, C)$. By definition of $W$, $c \nqvdash^i \Gamma^{G'} \not \subseteq p_j$ for every $j < 3$. Unfolding the definition of the forcing question, for every $j < 3$, the class
$$
\Ucal_C \cap \bigcap \{ \Ucal_{\zeta(e_j, \sigma^i \cup \rho, x)} : x \in \omega, \rho \subseteq X \cap A^i \}
$$
is a largeness class, where $e_j \in \omega$ is an index of the $\Sigma^0_2$ formula $(\exists x)(\forall y)\Phi_{e_j}(G, x, y)$ which holds if $\Gamma^{G'} \not \subseteq p_j$. Let $C_j = C \cup \{ \langle \zeta(e_j, \sigma^i \cup \rho, x), 0 \rangle : x \in \omega, \rho \subseteq X \cap A^i \}$. Although $\Ucal_{C_j}$ is a largeness class for every $j < 3$, in general, it is not true that there are $j_0 < j_1 < 3$ such that $\Ucal_{C_{j_0}} \cap \Ucal_{C_{j_1}}$ is a largeness class. However, the product $\Ucal_{C_{j_0}} \times \Ucal_{C_{j_1}}$ is a largeness class in the following generalized sense:

\begin{definition}
A class $\Acal \subseteq 2^\omega \times 2^\omega$ is a \emph{largeness class} if
\begin{itemize}
	\item[(1)] For every $\langle X_0, X_1 \rangle \in \Acal$ and $Y_0 \supseteq X_0$ and $Y_1 \supseteq X_1$, $\langle Y_0, Y_1 \rangle \in \Acal$
	\item[(2)] For every $k, \ell \in \omega$ and every $X_0 \cup \dots \cup X_{k-1} \supseteq \omega$, and $Y_0 \cup \dots \cup Y_{\ell-1} \supseteq \omega$, there is some $r < k$ and $s < \ell$ such that $\langle X_s, Y_\ell\rangle \in \Acal$.
\end{itemize}
\end{definition}

Note that by taking the common refinement of the two covers, one can replace item (2) by \qt{For every $k \in \omega$ and every $X_0 \cup \dots \cup X_{k-1} \supseteq \omega$, there is some $j_0, j_1 < k$ such that $\langle X_{j_0}, X_{j_1} \rangle \in \Acal$.} Again, given some pair $\langle X_0, X_1 \rangle$, we let $\Lcal_{\langle X_0, X_1\rangle}$ be the class of all pairs $\langle Y_0, Y_1\rangle$ such that $Y_0 \cap X_0$ and $Y_1 \cap X_1$ are both infinite. The notion of largeness class over an arbitrary product and the $\Lcal$ notation is defined accordingly.

The cartesian product of two largeness classes is again a largeness class. However, some largeness classes over $2^\omega \times 2^\omega$ cannot be expressed as a cartesian product of two largeness classes over $2^\omega$, as witnessed by the class $\{ \langle X, Y \rangle : |X \cap Y| = \infty \}$. The extension of $c$ forcing partiality of $\Gamma^{G'}$ on side $i$ is of the following type:
$$
c = (\sigma^0_{j_0,j_1}, \sigma^1_{j_0, j_1}, X_0, X_1, X_2, D : j_0 < j_1 < 3)
$$
where $\sigma^i_{j_0, j_1} \subseteq A^i$ for every $j_0 < j_1 < 3$ and $i < 2$,
$X_0, X_1, X_2$ are sets of low degree with $\max (\sigma^0_{j_0,j_1}, \sigma^1_{j_0,j_1}) < \min (X_{j_0}, X_{j_1})$ for every $j_0 < j_1 < 3$.
Moreover, $D$ is a $\Delta^0_2$ set of indices such that
$\Ucal^3_D$ is a largeness class over $2^\omega \times 2^\omega \times 2^\omega$, such that $\Ucal^3_D \subseteq \Lcal_{\langle X_0, X_1, X_2\rangle}$.

One can think of such a condition $c$ from a partial order $\Pb$ as three parallel conditions $c^{\{0,1\}}$, $c^{\{0,2\}}$ and $c^{\{1,2\}}$ from a partial order $\Qb$ where
$$
c^{\{j_0, j_1\}} = (\sigma^0_{j_0,j_1}, \sigma^1_{j_0, j_1}, X_{j_0}, X_{j_1}, \pi_{\{j_0, j_1\}}(\Ucal^3_D))
$$
with $\pi_{\{j_0, j_1\}}(\Ucal^3_D) = \{\langle Y_{j_0}, Y_{j_1}\rangle : \langle Y_0, Y_1, Y_2 \rangle \in \Ucal^3_D \}$. Any such $\Qb$-condition $c^{\{j_0, j_1\}}$ has two reservoirs $X_{j_0}$ and $X_{j_1}$, both of which $\sigma^0_{j_0,j_1}$ and $\sigma^1_{j_0, j_1}$ take elements from.
The forcing relation over $\Qb$ is defined as follows:

\begin{definition}
Given a $\Delta_0$ formula $\Phi_e(G, x, y)$, a finite set $\sigma \in 2^{<\omega}$ and some integer $x \in \omega$, let $\zeta_2(e, \sigma, x)$ be an index of the $\Sigma^0_1$ class
$$
\Ucal^2_{\zeta_2(e, \sigma, x)} = \{ \langle X_0, X_1 \rangle : (\exists \rho \subseteq X_0 \cup X_1)(\exists y \in \omega) \neg \Phi_e(\sigma \cup \rho, x, y) \}
$$
\end{definition}

\begin{definition}
Let $c = (\sigma^0, \sigma^1, X_0, X_1, \Acal)$ be a $\Qb$-condition, $i < 2$ and $\Phi_e(G, x, y)$ be a $\Delta_0$ formula.
\begin{itemize}
	\item[(a)] $c \Vdash^i (\exists x)(\forall y)\Phi_e(G,x, y)$ if there is some $x \in \omega$ such that $c \Vdash^i (\forall y)\Phi_e(G, x, y)$ holds.
	\item[(b)] $c \Vdash^i (\forall x)(\exists y)\neg \Phi_e(G,x, y)$ if for every $x \in \omega$ and every $\rho \subseteq (X_0 \cup X_1) \cap A^i$, $\Acal \subseteq \Ucal^2_{\zeta_2(e, \sigma^i \cup \rho, x)}$.
\end{itemize}
\end{definition}

We need again to define a notion of $i$-validity for a $\Qb$-condition to ensure that the forced $\Pi^0_2$ formulas hold for any sufficiently generic set.
A $\Qb$-condition $c = (\sigma^0, \sigma^1, X_0, X_1, \Acal)$ is \emph{$i$-valid} if $\langle X_0 \cap A^i, X_1 \cap A^i \rangle \in \Acal$. There is no reason to believe that every $\Qb$-condition must have a valid side. However, by the pigeonhole principle, for every $\Pb$-condition $c = (\sigma^0_{j_0,j_1}, \sigma^1_{j_0, j_1}, X_0, X_1, X_2, D : j_0 < j_1 < 3)$, there is some $j_0 < j_1 < 3$ and some $i < 2$ such that $c^{\{j_0, j_1\}}$ is an $i$-valid $\Qb$-condition. Indeed, since $\Ucal^3_C \subseteq \Lcal_{\langle X_0, X_1, X_2\rangle}$ is a largeness class, there are some $i_0, i_1, i_2 < 2$ such that $\langle X_0 \cap A^{i_0}, X_1 \cap A^{i_1}, X_2 \cap A^{i_2} \rangle \in \Ucal^3_C$. In particular, there are some $j_0 < j_2 < 3$ be such that $i_{j_0} = i_{j_1}$. The $\Qb$-condition $c^{\{j_0, j_1\}}$ is $i_{j_0}$-valid. The existence of a valid side is the reason we pick three pairwise incompatible valuations.

Having made the necessary definitions, consider the $\Pb$-condition
$$
d = (\sigma^0_{j_0,j_1}, \sigma^1_{j_0, j_1}, X_0, X_1, X_2, D : j_0 < j_1 < 3)
$$
where $\sigma^0_{j_0, j_1} = \sigma^0$, $\sigma^1_{j_0, j_1} = \sigma^1$,
$X_0 = X_1 = X_2 = X$ and $D$ is such that $\Ucal^3_D = \Ucal_{C_0} \times \Ucal_{C_1} \times \Ucal_{C_2}$. The condition $d$ is an extension of $c$
such that for every $j_0 < j_1 < 3$
$$
d^{\{j_0,j_1\}} \Vdash^i \Gamma^{G'}(n) \subseteq p_{j_0}
\mbox{ and }
d^{\{j_0,j_1\}} \Vdash^i \Gamma^{G'}(n) \subseteq p_{j_1}
$$

We need to define a new forcing question for the generalized notion of $\Pb$-condition, in order to satisfy more requirements. The new notion of $\Pb$-condition admits multiple branches, namely, the $\Qb$-conditions. Only one side of one branch is guaranteed to be valid. We therefore need to force the requirements on each side of each branch.

Given a $\Pb$-condition $c$ and some side $i < 2$, we let $H(c, i)$ be the set of branches (here the set of pairs $\{j_0, j_1\} \in [\{0, 1, 2\}]^2$) such that the requirement is not forced on the side $i$ of $c^{\{j_0, j_1\}}$. We design a forcing question parameterized by the set $H(c, i)$ such that if $c \qvdash^i_H (\exists x)(\forall y)\Phi_e(G, x, y)$ holds, then there is an extension $d \leq c$ which does not increase the number of branches, and such that $d^{\{j_0, j_1\}} \Vdash^i (\exists x)(\forall y)\Phi_e(G, x, y)$ for some $\{j_0, j_1\} \in H$. On the other hand, if  $c \nqvdash^i_H (\exists x)(\forall y)\Phi_{e_j}(G, x, y)$ for sufficiently many $e_j$ (which depends on the number of branches), then there is an extension $d \leq c$ which increases the number of branches,
but such that for every new branch $\nu$ refining a branch $\{j_0, j_1\} \in H$,
there are two indices $r \neq s$ such that $d^{[\nu]} \Vdash^i (\forall x)(\exists y)\neg \Phi_{e_r}(G, x, y)$ and $d^{[\nu]} \Vdash^i (\forall x)(\exists y)\neg \Phi_{e_s}(G, x, y)$ simultaneously.

In both cases, the number of branches for which the requirement is not forced decreases. In the first case, one more branch satisfies the $\Sigma^0_2$ outcome. In the second case, all the remaining branches satisfy the $\Pi^0_2$ outcome. Each time the $\Pi^0_2$ outcome occurs, then the number of branches increases.

\begin{definition}
Let $c = (\sigma^0_{j_0,j_1}, \sigma^1_{j_0, j_1}, X_0, X_1, X_2, C : j_0 < j_1 < 3)$ be a $\Pb$-condition, $i < 2$ and $\Phi_e(G, x, y)$ be a $\Delta_0$ formula. Fix $H \subseteq [\{0,1,2\}]^2$.
Let $c \qvdash^i_H (\exists x)(\forall y)\Phi_e(G, x, y)$ hold if
$$
\Ucal^3_C \cap \bigcap_{\{j_0, j_1\} \in H} \{ \langle X_0, X_1, X_2 \rangle :  \langle X_{j_0}, X_{j_1} \rangle \in \Ucal_{\zeta_2(e, \sigma^i_{j_0,j_1} \cup \rho, x)} : x \in \omega, \rho \subseteq (X_{j_0} \cap X_{j_1}) \cap A^i \}
$$
is not a largeness class.
\end{definition}

The $\Pb$ and $\Qb$ notions of forcing have to be generalized to conditions with arbitrarily many branches and reservoirs. The general case is formally defined and proven in the next section.

\begin{remark}
One important obstacle when using a variant of Mathias forcing to separate $\srt^2_2$ from $\coh$ on $\omega$-models is that every sufficiently generic filter produces a solution to $\srt^2_2$ which is r-cohesive \emph{as a set}. Indeed, given a condition $(\sigma^0, \sigma^1, X)$ and a computable set $R$, either $X \cap R$ or $X \cap \overline{R}$ is infinite (or belongs to some largeness class). Then either $(\sigma^0, \sigma^1, X \cap R)$ or $(\sigma^0, \sigma^1, X \cap \overline{R})$ is an extension forcing both sets to be either almost included in $R$ or in $\overline{R}$. In our situation, we overcome the problem by using increasingly many reservoirs simultaneously. Indeed, consider a condition $(\sigma^0, \sigma^1, X_0, X_1)$ where $\sigma^0$ and $\sigma^1$ take their elements from $X_0 \cup X_1$, and both $X_0$ and $X_1$ are required to be infinite. Then if only $X_0 \cap R$ and $X_1 \cap \overline{R}$ are infinite, the generic sets will have an infinite intersection with $R$ and $\overline{R}$.
\end{remark}

\section{Jump PA avoidance}\label{sect:delta2-jump-not-pa}

In this section, we give a formal proof of the following main theorem outlined in \Cref{sect:outline}:

\begin{theorem*}[\ref{thm:jump-pa-avpoidance-d22}]
For every set $Z$ whose jump is not of PA degree over $\emptyset'$ and every $\Delta^{0,Z}_2$ set $A$, there is an infinite subset $G \subseteq A$ or $G \subseteq \overline{A}$ such that $(G \oplus Z)'$ is not of PA degree over $\emptyset'$.
\end{theorem*}

Before proving \Cref{thm:jump-pa-avpoidance-d22}, we prove its main  consequence, namely, the separation of $\srt^2_2$ from $\coh$ on $\omega$-models.

\begin{theorem*}[\ref{thm:srt22-coh-omega-models}]
There is an $\omega$-model of $\rca \wedge \srt^2_2$ which is not a model of $\coh$.
\end{theorem*}
\begin{proof}
By \Cref{thm:jump-pa-avpoidance-d22}, there is a countable sequence of sets $Z_0, Z_1, \dots$ such that for every $s \in \omega$, the jump of $Z_0 \oplus \dots \oplus Z_s$ is not of PA degree over $\emptyset'$,
and for every $\Delta^0_2(Z_0 \oplus \dots \oplus Z_s)$ set $A$, there is some $t \in \omega$ such that $Z_t \subseteq A$ or $Z_t \subseteq \overline{A}$.
Let $\Ical = \{ X \in 2^\omega : (\exists s) X \leq_T Z_0 \oplus \dots \oplus Z_s \}$. The collection $\Ical$ is a Turing ideal. Let $\Mcal$ be the $\omega$-structure whose second-order part is $\Ical$. Every instance of $\srt^2_2$ in $\Ical$ has a solution in $\Ical$, so $\Mcal$ is an $\omega$-model of $\srt^2_2$. Moreover, $\Ical$ does not contain any set whose jump is of PA degree over $\emptyset'$. By Jockusch and Stephan~\cite{Jockusch1993cohesive}, $\Ical$ does not contain any p-cohesive set, so $\Mcal$ is not a model of $\coh$.
\end{proof}

The remainder of this section is devoted to the proof of  \Cref{thm:jump-pa-avpoidance-d22}. In what follows, fix a $\Delta^0_2$ set $A$, and let $A^0 = A$ and $A^1 = \overline{A}$. Fix also a countable Scott set $\Mcal$, coded by a low set~$M$, as in the previous section.

\subsection{Largeness classes}

The following notion of largeness class was introduced by the authors in~\cite{Monin2018Pigeons} to design a notion of forcing controlling the second jump of solutions to the pigeonhole principle.
In what follows, given two sets $A$ and $B$, we denote by $A \to B$ the class of all functions from $A$ to $B$.

\begin{definition}
Fix a finite set $I \subseteq \omega^{<\omega}$.
A \emph{largeness class} is a collection of sets $\Acal \subseteq I \to 2^\omega$ such that
\begin{itemize}
	\item[(a)] If $\langle X_\nu : \nu \in I\rangle \in \Acal$ and $Y_\nu \supseteq X_\nu$ for every $\nu \in I$, then $\langle Y_\nu : \nu \in I \rangle \in \Acal$
	\item[(b)] For every $k$-cover $Y_0, \dots, Y_{k-1}$ of $\omega$, there is some $\langle j_\nu < k : \nu \in I\rangle$ such that $\langle Y_{j_\nu} : \nu \in I \rangle \in \Acal$.
\end{itemize}
\end{definition}

Whenever $I = \{\epsilon\}$, we identify the class $\{\epsilon\} \to 2^\omega$ with the class $2^\omega$. This yields a notion of largeness class for subsets of $2^\omega$.
The collection of all the infinite sets is a largeness class. Moreover, any superclass of a largeness class is again a largeness class.

Given $I \subseteq 2^{<\omega}$, we fix a uniformly $M$-computable enumeration $\Ucal^I_0, \Ucal^I_1, \dots$ of all the $\Sigma^{0,Z}_1$ subclasses of $I \to 2^\omega$, upward-closed under the superset relation, where $Z \in \Mcal$. Here, the upward-closure means that if $\langle X_\nu : \nu \in I \rangle \in \Ucal^I_e$ and $Y_\nu \supseteq X_\nu$ for every $\nu \in I$, then $\langle Y_\nu : \nu \in I \rangle \in \Ucal^I_e$.
Given a set $C \subseteq \omega$, we write
$$
\Ucal^I_C = \bigcap_{e \in C} \Ucal^I_e
$$
If $C$ is $\Delta^0_2$, then $\Ucal^I_C$ is $\Pi^0_2$.

\begin{lemma}\label{lem:decreasing-largeness-yields-largeness}
Suppose $\Acal_0 \supseteq \Acal_1 \supseteq \dots$ is a decreasing sequence of largeness classes.
Then $\bigcap_s \Acal_s$ is a largeness class.
\end{lemma}
\begin{proof}
If $\langle X_\nu : \nu \in I \rangle \in \bigcap_s \Acal_s$ and $Y_\nu \supseteq X_\nu$ for every $\nu \in I$, then for every $s$, since $\Acal_s$ is a largeness class, $\langle Y_\nu : \nu \in Y \rangle \in \Acal_s$, so  $\langle Y_\nu : \nu \in Y \rangle \in \bigcap_s \Acal_s$.
Let $Y_0, \dots, Y_{k-1}$ be a $k$-cover of $\omega$. For every $s \in \omega$, there is some $\langle j_\nu < k : \nu \in I \rangle$
such that $\langle Y_{j_\nu} : \nu \in I \rangle \in \Acal_s$. By the infinite pigeonhole principle, there is some $\langle j_\nu < k : \nu \in I \rangle$ such that $\langle Y_{j_\nu} : \nu \in I \rangle \in \Acal_s$ for infinitely many $s$. Since $\Acal_0 \supseteq \Acal_1 \supseteq$ is a decreasing sequence,  $\langle Y_{j_\nu} : \nu \in I \rangle\in \bigcap_s \Acal_s$.
\end{proof}

\Cref{lem:decreasing-largeness-yields-largeness} has several very useful consequences. In particular, if $\Ucal_C$ is not a largeness class, then by \Cref{lem:decreasing-largeness-yields-largeness}, there is a finite set $F \subseteq C$ such that the class $\Ucal_F$  is not a largeness class. The set $F$ being finite, the class $\Ucal_F$ is $\Sigma^{0,Z}_1$ for some $Z \in \Mcal$, so since $\Mcal$ is a Scott set, there is a $k$-cover $X_0 \cup \dots \cup X_{k-1} = \omega$ in $\Mcal$ such that for every $j < k$, $X_j \not \in \Ucal_F \supseteq \Ucal_C$. Therefore we can always find a $k$-cover belonging to $\Mcal$, witnessing that $\Ucal_C$ is not a largeness class, whatever the complexity of the set~$C$. Another consequence is the following lemma.

\begin{lemma}\label{lem:largeness-class-complexity}
Let $\Acal$ be a $\Sigma^0_1$ class.
The sentence “$\Acal$ is a largeness class" is $\Pi^0_2$.
\end{lemma}
\begin{proof}
Say $\Acal = \{ \langle X_\nu : \nu \in I \rangle : (\exists \vec{\sigma} \preceq \vec{X})\varphi(\vec{\sigma}) \}$ where $\varphi$ is a $\Sigma^0_1$ formula.
By compactness, $\Acal$ is a largeness class iff for every $\vec{\sigma} = \langle \sigma_\nu : \nu \in I\rangle$ and $\vec{\tau} = \langle \tau_\nu : \nu \in I \rangle$ such that $\sigma_\nu \subseteq \tau_\nu$ for every $\nu \in I$ and $\varphi(\vec{\sigma})$ holds, $\varphi(\vec{\tau})$ holds, and for every $k$, there is some $n \in \omega$ such that for every $\sigma_0 \cup \dots \cup \sigma_{k-1} = \{0, \dots, n\}$, there is some $\langle j_\nu < k : \nu \in I \rangle$ such that $\varphi(\langle \sigma_{j_\nu} : \nu \in I \rangle)$ holds.
\end{proof}

\begin{definition}
Given $\langle X_\nu : \nu \in I\rangle$, we let
$$
\Lcal_{\langle X_\nu : \nu \in I \rangle} = \{ \langle Y_\nu : \nu \in I \rangle : (\forall \nu \in I) |Y_\nu \cap X_\nu| = \infty  \}
$$
\end{definition}

The following trivial lemma is very useful.

\begin{lemma}\label{lem:lcal-robust-finite-changes}
Let  $\langle X_\nu : \nu \in I\rangle$ and  $\langle Y_\nu : \nu \in I\rangle$
be such that $X_\nu =^{*} Y_\nu$ for every $\nu \in I$.
Then $\Lcal_{\langle X_\nu : \nu \in I\rangle} = \Lcal_{\langle Y_\nu : \nu \in I\rangle}$.
\end{lemma}
\begin{proof}
By symmetry, it suffices to prove that $\Lcal_{\langle X_\nu : \nu \in I\rangle} \subseteq \Lcal_{\langle Y_\nu : \nu \in I\rangle}$.
Fix $\langle Z_\nu : \nu \in I \rangle \in \Lcal_{\langle X_\nu : \nu \in I\rangle}$. Then for every $\nu \in I$, $Z_\nu \cap X_\nu$ is infinite.
Since $X_\nu =^{*} Y_\nu$, then $Z_\nu \cap Y_\nu$ is infinite, so $\langle Z_\nu : \nu \in I \rangle \in \Lcal_{\langle Y_\nu : \nu \in I\rangle}$.
\end{proof}

\begin{lemma}\label{lem:infinite-intersection-largeness}
Let $\Acal$ be a largeness class. The class
$$
\Lcal(\Acal) = \{ \langle X_\nu : \nu \in I \rangle \in \Acal : \Acal \cap \Lcal_{\langle X_\nu : \nu \in I \rangle} \mbox{ is a largeness class } \}
$$
is a largeness subclass of $\Acal$.
\end{lemma}
\begin{proof}
The class $\Lcal(\Acal)$ is trivially a subclass of $\Acal$. Moreover, $\Lcal(\Acal)$ is upward-closed.
Suppose for the sake of contradiction that $\Lcal(\Acal)$ is not a largeness class.
Then there is a cover $X_0 \cup \dots \cup X_{k-1} = \omega$ such that for every
$\langle j_\nu < k : \nu \in I\rangle$, $\langle X_{j_\nu} : \nu \in I \rangle \not \in \Lcal(\Acal)$. In other words, for every $\langle j_\nu < k : \nu \in I\rangle$, $\Acal \cap \Lcal_{\langle X_{j_\nu} : \nu \in I \rangle}$ is not a largeness class. Thus for every $\langle j_\nu < k : \nu \in I\rangle$, there is a cover $Y_0 \cup \dots \cup Y_{\ell-1}  = \omega$ such that for every $\langle i_\nu < \ell : \nu \in I \rangle$,
$\langle Y_{i_\nu} : \nu \in I \rangle \not \in \Acal \cap \Lcal_{\langle X_{j_\nu} : \nu \in I \rangle}$. By taking the common refinement of all these covers, there is a cover $Z_0 \cup \dots \cup Z_{n-1} = \omega$ such that
for every $\langle j_\nu < k : \nu \in I\rangle$, $\langle Z_{j_\nu} : \nu \in I \rangle \not \in \Acal \cap \Lcal_{\langle Z_{j_\nu} : \nu \in I \rangle}$. Since $\Acal$ is a largeness class, there is some $\langle j_\nu < k : \nu \in I\rangle$ such that $\langle Z_{j_\nu} : \nu \in I \rangle \in \Acal$. Then $\langle Z_{j_\nu} : \nu \in I \rangle \in \Lcal_{\langle Z_{j_\nu} : \nu \in I \rangle}$. Contradiction.
\end{proof}

\begin{lemma}\label{lem:intersection-compatibility}
If $\Acal \cap \Lcal_{\langle X_\nu : \nu \in I \rangle} \cap \Lcal_{\langle Y_\nu : \nu \in I \rangle}$ is a largeness class, then so is $\Acal \cap \Lcal_{\langle X_\nu \cap Y_\nu : \nu \in I \rangle}$.
\end{lemma}
\begin{proof}
$\Acal \cap \Lcal_{\langle X_\nu \cap Y_\nu : \nu \in I \rangle}$ is upward-closed.
Let $Z_0 \cup \dots \cup Z_{k-1} = \omega$. By refining the covering, we can assume that for every $j < k$ and $\nu \in I$, $Z_j \subseteq X_\nu$ or $Z_j \cap X_\nu = \emptyset$, and $Z_j \subseteq Y_\nu$ or $Z_j \cap Y_\nu = \emptyset$.
Since $\Acal \cap \Lcal_{\langle X_\nu : \nu \in I \rangle} \cap \Lcal_{\langle Y_\nu : \nu \in I \rangle}$ is a largeness class, there is some $\langle j_\nu < k : \nu \in I\rangle$ such that $\langle Z_{j_\nu} : \nu \in I \rangle \in \Acal \cap \Lcal_{\langle X_\nu : \nu \in I \rangle} \cap \Lcal_{\langle Y_\nu : \nu \in I \rangle}$.
We claim that $Z_{j_\nu} \subseteq X_\nu \cap Y_\nu$ for every $\nu \in I$.
Indeed, since $\langle Z_{j_\nu} : \nu \in I \rangle \in  \Lcal_{\langle X_\nu : \nu \in I \rangle}$, $Z_{j_\nu} \cap X_\nu \neq \emptyset$, so $Z_{j_\nu} \subseteq X_\nu$. Similarly, since $\langle Z_{j_\nu} : \nu \in I \rangle \in \Lcal_{\langle Y_\nu : \nu \in I \rangle}$, $Z_{j_\nu} \subseteq Y_\nu$.
Thus $\langle Z_{j_\nu} : \nu \in I \rangle \in \Acal \cap \Lcal_{\langle X_\nu \cap Y_\nu : \nu \in I \rangle}$.
\end{proof}

\begin{definition}
Given a class $\Acal \subseteq I \to 2^\omega$ and a set $J \subseteq I$, define the projection
$\pi_J(\Acal)$ be the set of all $\langle X_\nu : \nu \in J\rangle$ such that the class
$$
\{\langle X_\nu : \nu \in I - J \rangle : \langle X_\nu : \nu \in I\rangle \in \Acal\}
$$
is a largeness class.
\end{definition}

\begin{lemma}
If $\Acal \subseteq I \to 2^\omega$ is a largeness class and $J \subseteq I$, then $\pi_J(\Acal)$ is a largeness class.
\end{lemma}
\begin{proof}
The class $\pi_J(\Acal)$ is upward-closed by upward-closure of $\Acal$.
Let $Y_0 \cup \dots \cup Y_{k-1} = \omega$.
Suppose for the sake of contradiction that for every $\langle j_\nu : \nu \in J \rangle$,
$\langle Y_{j_\nu} : \nu \in J\rangle \not \in \pi_J(\Acal)$. Thus for every $\langle j_\nu : \nu \in J \rangle$, the class
$$
\{ \langle X_\nu : \nu \in I - J \rangle : \langle X_\nu : \nu \in I - J \rangle \cup \langle Y_{j_\nu} : \nu \in J \rangle \in \Acal \}
$$
is not a largeness class.
By taking the common refinement of $Y_0 \cup \dots \cup Y_{k-1} = \omega$ with all the covers of $\omega$ witnessing that the classes above are not largeness classes, we obtain a cover $Z_0 \cup \dots \cup Z_{\ell-1} = \omega$ witnessing that $\Acal$ is not a largeness class.
\end{proof}

\begin{lemma}\label{lem:largeness-projection-extension}
Let $\Ucal^I_C \subseteq I \to 2^\omega$ be a largeness class for some $\Delta^0_2$ set $C$,
and $\Acal \subseteq \pi_J(\Ucal^I_C)$ be a $\Pi^0_2$ largeness class.
Then there is a $\Delta^0_2$ set $D \supseteq C$ such that
$\Ucal^I_D \subseteq \Ucal^I_C$ is a largeness class and $\pi_J(\Ucal^I_D) = \Acal$.
\end{lemma}
\begin{proof}
Let $\Ucal^I_D$ be the class of all $\langle X_\nu : \nu \in I \rangle \in \Ucal^I_C$ such that $\langle X_\nu : \nu \in J \rangle \in \Acal$. Since $\Acal$ is a $\Pi^0_2$ class, then so is $\Ucal^I_D$, and therefore $D$ can be chosen to be $\Delta^0_2$. Furthermore we can assume without loss of generality that $D \supseteq C$. By construction, $\Ucal^I_D \subseteq \Ucal^I_C$.

We claim that $\Ucal^I_D$ is a largeness class.
$\Ucal^I_D$ is upward-closed since both $\Ucal^I_C$ and $\Acal$ are.
Let $Y_0 \cup \dots \cup Y_{k-1} = \omega$.
Since $\Acal$ is a largeness subclass of $J \to 2^\omega$, there is some $\langle j_\nu : \nu \in J \rangle$ such that $\langle Y_{j_\nu} : \nu \in J \rangle \in \Acal$. Since $\Acal \subseteq \pi_J(\Ucal^I_C)$, the collection
$$
\{ \langle X_\nu : \nu \in I - J \rangle : \langle X_\nu : \nu \in I - J \rangle \cup \langle Y_{j_\nu} : \nu \in J \rangle \in \Ucal^I_C \}
$$
is a largeness class. Therefore, there is some $\langle j_\nu : \nu \in I - J\rangle$ such that $\langle Y_{j_\nu} : \nu \in I \rangle \in \Ucal^I_C$.
In particular, $\langle Y_{j_\nu} : \nu \in I \rangle \in \Ucal^I_D$.
This proves that $\Ucal^I_D$ is a largeness class.

Last, it is immediate to see that $\pi_J(\Ucal^I_D) = \Acal$.
\end{proof}

\subsection{Valuation}

The notion of valuation is a combinatorial trick of Liu~\cite{Liu2012RT22} to obtain, whenever the $\Sigma^0_2$ outcome cannot be satisfied, arbitrarily many $\Pi^0_2$ formulas such that forcing any two of them is sufficient to force partiality of a Turing functional. We now define the notion of valuation and prove Liu's combinatorial lemma in its full generality (\Cref{lem:combi-liu-valuation}).

\begin{definition}
A \emph{valuation} is a partial function $p \subseteq \omega \to 2$.
A valuation is \emph{$X$-correct} if $p(n) = \Phi^{X}_n(n)$ for all $n \in \dom(p)$. Two valuations $p, q$ are \emph{incompatible} if there is an $n \in \dom(p) \cap \dom(q)$ such that $p(n) \neq q(n)$.
\end{definition}

The following lemma is adapted from Lemma 6.6 in Liu~\cite{Liu2012RT22}.

\begin{lemma}[Liu~\cite{Liu2012RT22}]\label{lem:combi-liu-valuation}
Fix $X$ and $Y \geq_T X$ such that $Y$ is not of PA degree relative to $X$.
Let $W$ be a $Y$-c.e.\  set of valuations.
Either $W$ contains an $X$-correct valuation, or for every $k$,
there are $k$ pairwise incompatible valuations outside of $W$.
\end{lemma}
\begin{proof}
Suppose $W$ does not contain any $X$-correct valuation, otherwise we are done.
Let $S$ be the collection of all finite sets $F$ such that for each $n \not \in F$, either $\Phi^X_n(n)\downarrow$ or there is a valuation $p \in W$ such that $F \cup \{n\} \subseteq \dom p$ and for every $m \in \dom p \setminus (F \cup \{n\})$, we have $p(m) = \Phi^X_m(m)\downarrow$. If $F \not \in S$, there there is at least one $n \not \in F$ such that the above does not hold. We say that any such $n$ \emph{witnesses} $F \not \in S$.

First suppose for the sake of contradiction that $\emptyset \in S$. Then for each $n$, either $\Phi^X_n(n)\downarrow$ or there is a valuation $p \in W$
such that $n \in \dom p$, and for every $m \in \dom p \setminus \{n\}$,
$\Phi^X_m(m)\downarrow$. We can then define a $Y$-computable completion $h$ of $n \mapsto \Phi^X_n(n)$ as follows. Given $n$, wait until either $\Phi^X_n(n)\downarrow$, in which case let $h(n) = \Phi^X_n(n)$, or a $p$ as above enters $W$, in which case we let $h(n) = 1-p(n)$. Since $W$ does not contain any $X$-correct valuation, in the latter case, if $\Phi^X_n(n)\downarrow$ then $\Phi^X_n(n) \neq p(n)\downarrow$, so $h(n) = \Phi^X_n(n)$. Since $Y$ is not of PA degree over $X$, this case cannot occur, so $\emptyset \not \in S$.

Let $n_0$ witness the fact that $\emptyset \not \in S$. Given $n_0, \dots, n_j$,
if $\{n_0, \dots, n_j\} \not \in S$, then let $n_{j+1}$ witness this fact. Note that if $n_j$ is defined, then $\Phi^X_{n_j}(n_j)\uparrow$.

Suppose for the sake of contradiction that $\{n_0, \dots, n_j\} \in S$. Then $\{n_0, \dots, n_{j-1}\} \not \in S$, otherwise $n_j$ would not be defined. We can then define a $Y$-computable completion $h$ of $n \mapsto \Phi^X_n(n)$ as follows. First, let $h(n_\ell) = 0$ for $\ell \leq j$. Given $n \not \in \{n_0, \dots, n_j\}$, we wait until either $\Phi^X_n(n)\downarrow$, in which case we let $h(n) = \Phi^X_n(n)$, or a valuation $p$ enters $W$
	such that $\{n_0, \dots, n_j, n\} \subseteq \dom p$ and for every $m \in \dom p \setminus \{n_0, \dots, n_j, n\}$, $p(m) = \Phi^X_m(m)\downarrow$. If $\Phi^X_n(n)\uparrow$, then the latter case must occur, since $\{n_0, \dots, n_j\} \in S$. In this case, we cannot have $p(n) = \Phi^X_n(n)$, as then $p$ would be a counter-example to the fact that $n_j$ witnesses $\{n_0, \dots, n_{j-1}\} \not \in S$. Thus we can let $h(n) = 1-p(n)$. We again have a contradiction since  $Y$ is not of PA degree over $X$.

Thus $\{n_0, \dots, n_j\} \not \in S$ for all $j$. There are $2^{j+1}$ pairwise incompatible valuations with domain $\{n_0, \dots, n_j\}$. None of them can be in $W$, as this would contradict the fact that $n_j$ witnesses $\{n_0, \dots, n_{j-1}\} \not \in S$. This completes the proof.
\end{proof}

\begin{lemma}\label{lem:jump-pa-reformulated}
Let $G$ be a set such that for every $\Delta_0$ formula $\Phi_e(G, x, y, p)$, where $x$ and $y$ are integer variables and $p$ is a valuation variable,
either $(\exists x)(\forall y)\Phi_e(G, x, y, p)$
holds for some $\emptyset'$-correct valuation $p$, or $(\forall x)(\exists y)\neg \Phi_e(G, x, y, p_0)$ and $(\forall x)(\exists y)\neg \Phi_e(G, x, y, p_0)$
hold for two incompatible valuations $p_0$ and $p_1$.
Then $G'$ is not PA over $\emptyset'$.
\end{lemma}
\begin{proof}
Suppose for the sake of contradiction that $G'$ is of PA degree over $\emptyset'$. In particular, there is a Turing functional $\Gamma$ such that $\Gamma^{G'}$ is a $\{0,1\}$-valued completion of $n \mapsto \Phi^{\emptyset'}_n(n)$. Given $n, s \in \omega$, we denote by $\Gamma^{G'}(n)[s]$ the $G$-computable $s$-approximation of $\Gamma^{G'}(n)$.
Let $\Phi_e(G, x, y, p)$ hold if  there is some $n \in \dom p$ such that if $\Gamma^{G'}(n)[x+y]\downarrow$ then $\Gamma^{G'}(n)[x+y] \neq p(n)$. We have two cases:

Case 1:  $(\exists x)(\forall y)\Phi_e(G, x, y, p)$ holds for some $\emptyset'$-correct valuation $p$. Then there is some $n \in \dom(p)$ such that $\Gamma^{G'}(n)\uparrow$ or $\Gamma^{G'}(n)\downarrow \neq p(n)$. By definition of a $\emptyset'$-correct valuation, $\Phi^{\emptyset'}_n(n)\downarrow = p(n)$ for every $n \in \dom(p)$. Thus  $\Gamma^{G'}$ is not a completion of $n \mapsto \Phi^{\emptyset'}_n(n)$.

Case 2: $(\forall x)(\exists y)\neg \Phi_e(G, x, y, p_0)$ and $(\forall x)(\exists y)\neg \Phi_e(G, x, y, p_0)$  hold for two incompatible valuations $p_0$ and $p_1$. Then $\Gamma^{G'} \upharpoonright \dom p_0 \subseteq p_0$ and $\Gamma^{G'} \upharpoonright \dom p_1 \subseteq p_1$. Since $p_0$ and $p_1$ are incompatible, then $\Gamma^{G'}$ is partial.
\end{proof}

\subsection{Index set}

We now define an ordered structure of sets of indices to simplify branches refinement whenever the $\Pi^0_2$ outcome occurs.
Define the sequence of integers $u_0, u_1, \dots$ inductively by $u_0 = 1$ and $u_{n+1} =  {2u_n+1 \choose 2} u_n$.

\begin{definition}
Given $n \in \omega$, the \emph{$n$-index set} is defined inductively as follows.
The 0-index set $\Ical_0$ is a the singleton empty string $\{\epsilon\}$.
Let $\Ical_n$ be the $n$-index set. The $(n+1)$-index set is the set
$$
\Ical_{n+1} = (2u_n+1) \times \Ical_n = \{ x^\frown \nu : x \leq 2u_n, \nu \in \Ical_n \}
$$
\end{definition}

\begin{definition}
Given $n \in \omega$, an \emph{$n$-index} is defined inductively as follows.
The unique 0-index is the singleton empty string $\{\epsilon\}$.
Given an $n$-index $I \subseteq \Ical_n$, an $(n+1)$-index
is a set $\{ x^\frown \nu : \nu \in I \} \cup \{ y^\frown \nu : \nu \in I \}$ for some $x < y \leq 2u_n$.
\end{definition}

Note that in particular, an $n$-index is a subset of $\Ical_n$. We write $I \lhd \Ical_n$ to say that $I$ is an $n$-index.

\begin{lemma}\label{lem:counting-n-indices}
For every $n \in \omega$, $|\{I \subseteq \Ical_n : I \lhd \Ical_n \}| = u_n$.
\end{lemma}
\begin{proof}
By induction over $n$. Case $n = 0$. There is only one 0-index and $u_0 = 1$.
Suppose $|\{I \subseteq \Ical_n : I \lhd \Ical_n \}| = u_n$.
Then $|\{J \subseteq \Ical_{n+1} : J \lhd \Ical_{n+1} \}| = |{2u_n+1 \choose 2}| \cdot |\{I \subseteq \Ical_n : I \lhd \Ical_n \}| = {2u_n+1 \choose 2} u_n = u_{n+1}$.
\end{proof}

The following lemma is the main combinatorial lemma of indices, which will be used in \Cref{lem:every-condition-valid-side} to prove that every $\Pb$-condition admits a branch with a valid side.

\begin{lemma}\label{lem:pigeonhole-indices}
For every $n \in \omega$ and every 2-cover $B_0 \cup B_1 = \Ical_n$,
there is an $n$-index $I \lhd \Ical_n$ and some $i < 2$ such that $I \subseteq B_i$.
\end{lemma}
\begin{proof}
By induction on $n$. The case $n = 0$ is trivial.
Assume it holds for $n$. We prove it for $n+1$.
For every $x \leq 2u_n$ and $i < 2$, let $B_{x,i} = \{ \nu : x^\frown \nu \in B_i \}$. By induction hypothesis, there is some $I_x \lhd \Ical_n$ and $i_x < 2$ such that $I_x \subseteq B_{x,i_x}$. By Lemma~\ref{lem:counting-n-indices}, $|\{I \subseteq \Ical_n : I \lhd \Ical_n \}| = u_n$ so by the pigeonhole principle, there is some $x < y \leq 2u_n$ such that $I_x = I_y$ and $i_x = i_y$. The $(n+1)$-index $\{ x^\frown \nu : \nu \in I_x \} \cup \{ y^\frown \nu : \nu \in I_x\}$ is included in $B_{i_x}$.
\end{proof}


\begin{definition}
Fix $m \geq n$, $J \lhd \Ical_m$ and $I \lhd \Ical_n$.
\begin{itemize}
	\item[(1)] Define a partial order $J \leq I$ inductively on $m-n$ as follows:
	If $m = n$, then $J \leq I$ if $J = I$.
	If $m > n$, then $J \leq I$ if $K \leq I$ for some $K \lhd \Ical_{m-1}$
	and there are some $x < y \leq 2u_{m-1}$ such that $J = \{ x^\frown \nu : \nu \in K \} \cup \{ y^\frown \nu : \nu \in K \}$.

	\item[(2)] Let $J \bowtie I$ be the set of all $\mu \in \omega^{<\omega}$ such that
	$I = \{ \nu : \mu^\frown \nu \in J \}$.

	\item[(3)] Given a class $\Acal \subseteq I \to 2^\omega$,
	let $J \otimes \Acal$ be the subclass of $J \to 2^\omega$ of all
	$\langle X^{\mu}_\nu : \nu \in I, \mu \in J \bowtie I \rangle$
	such that for every $\mu \in J \bowtie I$, $\langle X^{\mu}_\nu : \nu \in I \rangle \in \Acal$.

	\item[(4)] Given a class $\Acal \subseteq I \to 2^\omega$,
	let $\Ical_n \odot \Acal$ be the subclass of $\Ical_n \to 2^\omega$ of all
	$\langle X_\nu : \nu \in \Ical_n \rangle$ such that $\langle X_\nu : \nu \in I \rangle \in \Acal$.

	\item[(5)] Let $\Acal \subseteq I \to 2^\omega$ and $\Bcal \subseteq J \to 2^\omega$. We write $\Bcal \leq \Acal$ if $J \leq I$
	and $\Bcal \subseteq J \otimes \Acal$.
\end{itemize}
\end{definition}

One can easily prove that the relations $J \leq I$ and $\Bcal \leq \Acal$ are partial orders.


\begin{lemma}\label{lem:bowtie-covers-everything}
Suppose $J \leq I$. Then $J = \{ \mu^\frown \nu : \mu \in  J \bowtie I, \nu \in I \}$.
\end{lemma}
\begin{proof}
Say $J \lhd \Ical_n$ and $I \lhd \Ical_n$ with $m \geq n$.
We prove the lemma by induction over $m - n$.
Suppose $m = n$. Then $I = J$, so $J \bowtie I = \{\epsilon\}$. In particular, $J = \{\epsilon^\frown \nu : \nu \in I\}$.
Suppose $m > n$. By definition of $J \leq I$, there is some $K \lhd \Ical_{m-1}$ and $x < y \leq 2u_{m-1}$ such that $K \leq I$ and $J = \{ x^\frown \nu : \nu \in K \} \cup \{ y^\frown \nu : \nu \in K \}$. By induction hypothesis, $K = \{ \mu^\frown \nu : \mu \in  K \bowtie I, \nu \in I \}$. Thus $J \bowtie I = \{x^\frown \mu : \mu \in K \bowtie I\} \cup \{x^\frown \mu : \mu \in K \bowtie I\}$, and $J = \{ x^\frown \mu^\frown \nu : \mu \in K \bowtie I, \nu \in I\} \cup  \{ y^\frown \mu^\frown \nu : \mu \in K \bowtie I, \nu \in I\} = \{ \mu^\frown \nu : \mu \in  J \bowtie I, \nu \in I \}$.
\end{proof}

\subsection{$\Qb$-forcing}

We now define the partial order of $\Qb$-conditions, which represent branches of $\Pb$-conditions.

\begin{definition}
A $\Qb_n$-condition is a tuple
$(\sigma^0, \sigma^1, X_\nu, \Acal  : \nu \in I)$ where
\begin{itemize}
	\item[(1)] $\sigma^i \subseteq A^i$ for each $i < 2$ ; $I$ is an $n$-index
	\item[(2)] $\Acal \subseteq I \to 2^\omega$ is a largeness subclass of $\Lcal_{\langle X_\nu : \nu \in I \rangle}$
	\item[(3)] $X_\nu \in \Mcal$ for each $\nu \in I$ and $\Acal$ is $\Pi^0_2$
\end{itemize}
\end{definition}

We let $\Qb = \bigcup_n \Qb_n$.

\begin{definition}
The partial order on $\Qb$ is defined by
$$
(\tau^0, \tau^1, Y_\mu, \Bcal  : \mu \in J) \leq (\sigma^0, \sigma^1, X_\nu, \Acal  : \nu \in I)
$$
if $J \leq I$, for every $\mu \in J$ and $\nu \in I$ such that $\nu$ is a suffix of $\mu$, $Y_\mu \subseteq X_\nu$, $\Bcal \leq \Acal$, and for every $i < 2$,
$\sigma^i \preceq \tau^i$ and $\tau^i - \sigma^i \subseteq \bigcup_{\nu \in I} X_\nu$.
\end{definition}

\begin{lemma}\label{lem:qb-extension-compatible-Mathias-extension}
Let $c = (\sigma^0, \sigma^1, X_\nu, \Acal  : \nu \in I) \in \Qb_n$
and $d = (\tau^0, \tau^1, Y_\mu, \Bcal  : \mu \in J) \in \Qb_m$ with $m \geq n$
be such that $d \leq c$. Then for every $i < 2$,
$(\tau^i, \bigcup_{\mu \in J} Y_\mu)$ Mathias extends $(\sigma^i, \bigcup_{\nu \in I} X_\nu)$.
\end{lemma}
\begin{proof}
Since $J \leq I$, then by \Cref{lem:bowtie-covers-everything}, $J = \{\rho^\frown \nu : \rho \in J \bowtie I, \nu \in I\}$.
It follows that for every $\mu \in J$, there is some $\nu \in I$ such that $\nu$ is a suffix of $\mu$, and by definition of $d \leq c$, $Y_\mu \subseteq X_\nu$. Therefore $\bigcup_{\mu \in J} Y_\mu \subseteq \bigcup_{\nu \in I} X_\nu$.
Since $\tau^i \succeq \sigma^i$ and $\tau^i - \sigma^i \subseteq \bigcup_{\nu \in I} X_\nu$, then $(\tau^i, \bigcup_{\mu \in J} Y_\mu)$ Mathias extends $(\sigma^i, \bigcup_{\nu \in I} X_\nu)$.
\end{proof}

\subsection{Forcing relation}

\begin{definition}
Let $(\sigma, X)$ be a Mathias condition and $\Phi_e(G, x)$ be a $\Delta_0$ formula with an integer variable $x$.
\begin{itemize}
	\item[(1)] $(\sigma, X) \Vdash (\exists x)\Phi_e(G, x)$ if there is some $x \in \omega$ such that $\Phi_e(\sigma, x)$ holds.
	\item[(2)] $(\sigma, X) \Vdash (\forall x)\neg \Phi_e(G, x)$ if for every $x \in \omega$ and $\rho \subseteq X$, $\Phi_e(\sigma, x)$ does not hold.
\end{itemize}
\end{definition}

\begin{definition}
Given some $n$-index $I$, let $\zeta_I$ be the function which takes as a paramter an index $e$ of a $\Delta_0$ formula $\Phi_e(G, x, y)$,
a finite set $\sigma \in 2^{<\omega}$ and some integer $x \in \omega$, and returns a code for the $\Sigma^0_1$ class
$$
\Ucal^I_{\zeta_I(e, \sigma, x)} = \{ \langle X_\nu : \nu \in I \rangle : (\sigma, \bigcup_{\nu \in I} X_\nu) \not \Vdash (\forall y)\Phi_e(G, x, y) \}
$$
\end{definition}

\begin{definition}
Let $c = (\sigma^0, \sigma^1, X_\nu, \Acal  : \nu \in I)$ be a $\Qb_n$-condition, $\Phi_e(G, x, y)$ be a $\Delta_0$ formula with free integer variables $x$ and $y$, and let $i < 2$.
\begin{itemize}
	\item[(1)] $c \Vdash^i (\exists x)(\forall y)\Phi_e(G, x, y)$ if there is some $x \in \omega$ such that $(\sigma^i, \bigcup_{\nu \in I} X_\nu) \Vdash (\forall y)\Phi_e(G, x, y)$
	\item[(2)] $c \Vdash^i (\forall x)(\exists y)\neg \Phi_e(G, x, y)$ if for every $x \in \omega$, every $\rho \subseteq A^i \cap \bigcup_{\nu \in I} X_\nu$,
		$\Acal \subseteq \Ucal^I_{\zeta_I(e, \sigma^i \cup \rho, x)}$
\end{itemize}
\end{definition}

\begin{lemma}\label{lem:forcing-relation-extension}
Let $c, d$ be two $\Qb$-conditions such that $d \leq c$, and $\Phi_e(G, x, y)$ be a $\Delta_0$ formula.
\begin{itemize}
	\item[(1)] If $c \Vdash^i (\exists x)(\forall y)\Phi_e(G, x, y)$ then so does $d$.
	\item[(2)] If $c \Vdash^i (\forall x)(\exists y)\neg \Phi_e(G, x, y)$ then so does $d$.
\end{itemize}
\end{lemma}
\begin{proof}
Say $c = (\sigma^0, \sigma^1, X_\nu, \Acal  : \nu \in I) \in \Qb_n$
and $d = (\tau^0, \tau^1, Y_\mu, \Bcal  : \mu \in J) \in \Qb_m$ with $m \geq n$.

(1) Suppose $c \Vdash^i (\exists x)(\forall y)\Phi_e(G, x, y)$. Then there some $x \in \omega$ such that $(\sigma^i, \bigcup_{\nu \in I} X_\nu) \Vdash (\forall y)\Phi_e(G, x, y)$. By \Cref{lem:qb-extension-compatible-Mathias-extension}, $(\tau^i, \bigcup_{\mu \in J} Y_\mu)$ Mathias extends $(\sigma^i, \bigcup_{\nu \in I} X_\nu)$, so  $(\tau^i, \bigcup_{\mu \in J} Y_\mu) \Vdash (\forall y)\Phi_e(G, x, y)$. Folding the definition, $d \Vdash^i (\exists x)(\forall y)\Phi_e(G, x, y)$.

(2)
Fix some $x \in \omega$ and some $\rho \subseteq A^i \cap (\bigcup_{\mu \in J} Y_\mu)$. Since $d \leq c$, then there is some $\rho_0 \subseteq A^i \cap \bigcup_{\nu \in I} X_\nu$ such that $\tau^i = \sigma^i{}^\frown \rho_0$.
Moreover, by \Cref{lem:qb-extension-compatible-Mathias-extension}, $(\tau^i, \bigcup_{\mu \in J} Y_\mu)$ Mathias extends $(\sigma^i, \bigcup_{\nu \in I} X_\nu)$, so $\bigcup_{\mu \in J} Y_\mu \subseteq \bigcup_{\nu \in I} X_\nu$.
Therefore $\rho_0 \cup \rho \subseteq A^i \cap \bigcup_{\nu \in I} X_\nu$.
By applying the definition of $c \Vdash^i (\forall x)(\exists y)\neg \Phi_e(G, x, y)$ to $x$ and $\rho_0 \cup \rho$, $\Acal \subseteq \Ucal^I_{\zeta_I(e, \sigma^i \cup \rho_0 \cup \rho, x)} = \Ucal^I_{\zeta_I(e, \tau^i \cup \rho, x)}$.

We claim that $\Bcal \subseteq \Ucal^J_{\zeta_J(e, \tau^i \cup \rho, x)}$.
Since $\Bcal \leq \Acal$, $\Bcal \subseteq J \otimes \Acal$. Fix some $\langle Z^\mu_\nu : \nu \in I, \mu \in J \bowtie I\rangle$ such that for every $\mu \in J \bowtie I$, $\langle Z^\mu_\nu : \nu \in I\rangle \in \Acal$.
Since $\Acal \subseteq \Ucal^I_{\zeta_I(e, \tau^i \cup \rho, x)}$, for every $\mu \in J \bowtie I$, $(\tau^i, \bigcup_{\nu \in I} Z^\mu_\nu) \not \Vdash (\forall y)\Phi_e(G, x, y)$. Therefore $(\tau^i, \bigcup_{\nu \in I, \mu \in J \bowtie I} Z^\mu_\nu) \not \Vdash (\forall y)\Phi_e(G, x, y)$. Thus  $\langle Z^\mu_\nu : \nu \in I, \mu \in J \bowtie I\rangle \in \Ucal^J_{\zeta_J(e, \tau^i \cup \rho, x)}$. So $\Bcal \subseteq \Ucal^J_{\zeta_J(e, \tau^i \cup \rho, x)}$.
It follows that $d \Vdash^i (\forall x)(\exists y)\neg \Phi_e(G, x, y)$.
\end{proof}

\subsection{$\Pb$-forcing}

\begin{definition}
A $\Pb_n$-condition is a tuple
$(\sigma^0_I, \sigma^1_I, X_\nu, C  : I \lhd \Ical_n, \nu \in \Ical_n)$ where
\begin{itemize}
	\item[(1)] $\sigma^i_I \subseteq A^i$ for each $i < 2$  and $I \lhd \Ical_n$
	\item[(2)] $\Ucal^{\Ical_n}_C \subseteq \Ical_n \to 2^\omega$ is a largeness subclass of $\Lcal_{\langle X_\nu : \nu \in \Ical_n \rangle}$
	\item[(3)] $X_\nu \in \Mcal$ for each $\nu \in \Ical_n$ ; $C$ is $\Delta^0_2$
\end{itemize}
\end{definition}

A $\Pb_n$-condition $c = (\sigma^0_I, \sigma^1_I, X_\nu, C  : I \lhd \Ical_n, \nu \in \Ical_n)$ represents $u_n$ many parallel $\Qb_n$-conditions defined for each $I \lhd \Ical_n$ by
$$
c^{[I]} = (\sigma^0_I, \sigma^1_I, X_\nu, \pi_I(\Ucal_C)  : \nu \in I)
$$

We let $\Pb = \bigcup_n \Pb_n$.

\begin{definition}
The partial order on $\Pb$ is defined by
$$
(\tau^0_J, \tau^1_J, Y_\mu, D  : J \lhd \Ical_m, \mu \in \Ical_m) \leq (\sigma^0_I, \sigma^1_I, X_\nu, C  : I \lhd \Ical_n, \nu \in \Ical_n)
$$
if $m \geq n$, and for every $J \lhd \Ical_m$ and $I \lhd \Ical_n$ such that $J \leq I$
$$
(\tau^0_J, \tau^1_J, Y_\mu, \pi_J(\Ucal_D)  : \mu \in J) \leq (\sigma^0_I, \sigma^1_I, X_\nu, \pi_I(\Ucal_C)  : \nu \in I)
$$
\end{definition}

\begin{lemma}\label{lem:compatibility-extension-qb-pb}
Fix a $\Pb_n$-condition $c$ and some $I \lhd \Ical_n$.
For every $\Qb_n$-condition $d \leq c^{[I]}$, then
there is a $\Pb_n$-condition $e \leq c$ such that $e^{[I]} = d$.
\end{lemma}
\begin{proof}
Say $c = (\sigma^0_I, \sigma^1_I, X_\nu, C  : I \lhd \Ical_n, \nu \in \Ical_n)$
and $d = (\tau^0_I, \tau^1_I, Y_\nu, \Acal  : \nu \in I)$.
By \Cref{lem:largeness-projection-extension}, there is some $\Delta^0_2$ set $D \supseteq C$ such that $\Ucal^{\Ical_n}_D \subseteq \Ucal^{\Ical_n}_C$ is a largeness class and $\pi_I(\Ucal^{\Ical_n}_D) = \Acal$.
For every $J \lhd \Ical_n$ with $J \neq I$, let $\tau^0_J = \sigma^0_J$ and $\tau^1_J = \sigma^1_J$. For $\nu \in \Ical_n - I$, let $Y_\nu = X_\nu$.
The $\Pb_n$-condition $e = (\tau^0_J, \tau^1_J, Y_\nu, D  : J \lhd \Ical_n, \nu \in \Ical_n)$
is an extension of $c$ such that $e^{[I]} = d$.
\end{proof}

\subsection{Validity}

As explained in \Cref{sect:outline}, the forcing relation for $\Pi^0_2$ formulas relies on the 1-genericity of the filter for the properties to actually hold. We define the notion of validity so that the forced $\Pi^0_2$ formulas will be satisfied on the valid sides.

\begin{definition}
A $\Qb_n$-condition $(\sigma^0, \sigma^1, X_\nu, \Acal  : \nu \in I)$ is \emph{$i$-valid} for $i < 2$ if $\langle X_\nu \cap A^i : \nu \in I\rangle \in \Acal$.
\end{definition}

The following lemma ensures that whenever a $\Pi^0_2$ formula is forced on a valid side, then seeing the formula as a collection of $\Sigma^0_1$ formulas, one can satisfy each of them independently.

\begin{lemma}\label{lem:valid-side-progress-pi02}
Let $c$ be an $i$-valid $\Qb_n$-condition, and let $\Phi_e(G, x, y)$ be a $\Delta^0_1$ formula.
If $c \Vdash^i (\forall x)(\exists y)\neg \Phi_e(G, x, y)$ then for every $x \in \omega$ there is some $d = (\tau^0, \tau^1, Y_\nu, \Acal  : \nu \in I) \in \Qb_n$ extending $c$ such that $(\tau^i, \bigcup_{\nu \in I} Y_\nu) \Vdash (\exists y)\neg \Phi_e(G, x, y)$.
\end{lemma}
\begin{proof}
Say $c = (\sigma^0, \sigma^1, X_\nu, \Acal  : \nu \in I)$ and fix $x \in \omega$.
Since $c \Vdash^i (\forall x)(\exists y)\neg \Phi_e(G, x, y)$,
then $\Acal \subseteq \Ucal^I_{\zeta_I(e, \sigma^i, x)}$. Since $c$ is $i$-valid, then $\langle X_\nu \cap A^i : \nu \in I\rangle \in \Acal$, then $(\sigma^i, A^i \cap \bigcup_{\nu \in I} X_\nu) \not \Vdash (\forall y)\Phi_e(G, x, y)$.
Therefore there is some $\rho \subseteq A^i \cap \bigcup_{\nu \in I} X_\nu$ and some $y \in \omega$ such that $\neg \Phi_e(\sigma^i \cup \rho, x, y)$ holds.
Let $\tau^i = \sigma^i \cup \rho$ and $\tau^{1-i} = \sigma^{1-i}$.
For every $\nu \in I$, let $Y_\nu = X_\nu - \{0, \dots, \max \rho \}$.
By \Cref{lem:lcal-robust-finite-changes}, $\Acal \subseteq \Lcal_{\langle Y_\nu : \nu \in I\rangle}$.
The tuple $d = (\tau^0, \tau^1, Y_\nu, \Acal  : \nu \in I)$ is a $\Qb_n$-condition extending $c$. Moreover $(\tau^i, \bigcup_{\nu \in I} Y_\nu) \Vdash (\exists y)\neg \Phi_e(G, x, y)$.
\end{proof}

The following two lemmas state that every $\Pb$-filter induces as tree of valid sides.

\begin{lemma}\label{lem:every-condition-valid-side}
For every $\Pb_n$-condition $c$, there is some $I \lhd \Ical_n$ and some $i < 2$
such that $c^{[I]}$ is $i$-valid.
\end{lemma}
\begin{proof}
Say $c = (\sigma^0_I, \sigma^1_I, X_\nu, C  : I \lhd \Ical_n, \nu \in \Ical_n)$.
Since $A^0 \cup A^1 = \omega$ and by \Cref{lem:infinite-intersection-largeness}, $\Lcal(\Ucal^{\Ical_n}_C)$ is a largeness class, then there is some $\langle i_\nu < 2 : \nu \in \Ical_n \rangle$ such that $\langle A^{i_\nu} : \nu \in \Ical_n \rangle \in \Lcal(\Ucal^{\Ical_n}_C)$.
Thus $\Ucal^{\Ical_n}_C \cap \Lcal_{\langle X_\nu : \nu \in \Ical_n \rangle}  \cap \Lcal_{\langle A^{i_\nu} : \nu \in \Ical_n \rangle}$ is a largeness class,
so by \Cref{lem:intersection-compatibility}, $\Ucal^{\Ical_n}_C \cap  \Lcal_{\langle X_\nu \cap A^{i_\nu} : \nu \in \Ical_n \rangle}$ is a largeness class. 

Let $B_0 = \{\nu \in \Ical_n : i_\nu = 0 \}$ and $B_1 = \{\nu \in \Ical_n : i_\nu = 1 \}$. Since $B_0 \cup B_1 = \Ical_n$, by \Cref{lem:pigeonhole-indices}, there is some $I \lhd \Ical_n$ and some $i < 2$ such that $I \subseteq B_i$.
Since $\Ucal^{\Ical_n}_C \cap  \Lcal_{\langle X_\nu \cap A^{i_\nu} : \nu \in \Ical_n \rangle}$ is a largeness class, then 
	$\langle X_\nu \cap A^{i_\nu} : \nu \in I \rangle \in \pi_I(\Ucal^{\Ical_n}_C \cap \Lcal_{\langle X_\nu \cap A^{i_\nu} : \nu \in \Ical_n \rangle})$.
	Moreover $\pi_I(\Ucal^{\Ical_n}_C \cap \Lcal_{\langle X_\nu \cap A^{i_\nu} : \nu \in \Ical_n \rangle}) \subseteq \pi_I(\Ucal^{\Ical_n}_C)$, then 
	$\langle X_\nu \cap A^{i_\nu} : \nu \in I \rangle \in \pi_I(\Ucal^{\Ical_n}_C)$. As $I \subseteq B_i$, 
$\langle X_\nu \cap A^i : \nu \in I \rangle = \langle X_\nu \cap A^{i_\nu} : \nu \in I \rangle \in \pi_I(\Ucal^{\Ical_n}_C)$.
Thus the $\Qb_n$-condition $c^{[I]}$ is $i$-valid.
\end{proof}

\begin{lemma}\label{lem:valid-upward-closed}
Let $d, c \in \Qb$ be such that $d \leq c$.
If $d$ is $i$-valid, then so is $c$.
\end{lemma}
\begin{proof}
Say $c = (\sigma^0, \sigma^1, X_\nu, \Acal  : \nu \in I) \in \Qb_n$
and $d = (\tau^0, \tau^1, Y_\mu, \Bcal  : \mu \in J) \in \Qb_m$ with $m \geq n$.
Since $d$ is $i$-valid, $\langle Y_\mu \cap A^i : \mu \in J\rangle \in \Bcal$.
Since $d \leq c$, then $J \leq I$ and $\Bcal \leq \Acal$. By definition of $\Bcal \leq \Acal$, $\Bcal \subseteq J \otimes \Acal$, thus letting $\rho \in J \bowtie I$, and $Z_\nu = Y_{\rho^\frown \nu}$, $\langle Z_\nu \cap A^i : \nu \in I \rangle \in \Acal$. By upward-closure of $\Acal$, $\langle X_\nu \cap A^i : \nu \in I \rangle \in \Acal$. Thus $c$ is $i$-valid.
\end{proof}

The following lemma states that the generic sets corresponding to valid sides are infinite.

\begin{lemma}\label{lem:valid-side-infinite}
For every $i$-valid $\Qb_n$-condition $c = (\sigma^0, \sigma^1, X_\nu, \Acal  : \nu \in I)$, there is a $\Qb_n$-condition $d = (\tau^0, \tau^1, Y_\nu, \Acal  : \nu \in I) \leq c$ such that $\# \tau^i > \# \sigma^i$.
\end{lemma}
\begin{proof}
By definition of $i$-validity of $c$, $\langle X_\nu \cap A^i : \nu \in I \rangle \in \Acal \subseteq \Lcal_{\langle X_\nu : \nu \in I \rangle}$. So in particular, $X_\nu \cap A^i$ is infinite. Pick any $x \in \bigcup_{\nu \in I} X_\nu \cap A^i$, and let $Y_\nu = X_\nu - \{0, \dots, x\}$.
By \Cref{lem:lcal-robust-finite-changes}, $\Acal \subseteq \Lcal_{\langle Y_\nu : \nu \in I\rangle}$.
Then  $d = (\sigma^i \cup \{x\}, \sigma^{1-i}, Y_\nu, \Acal  : \nu \in I)$ is the desired extension.
\end{proof}

\subsection{Forcing question}

As explained in \Cref{sect:outline}, a $\Pb$-condition representing multiple parallel $\Qb$-condition, only one of which being valid on one side, we need to force the requirements on each side of each branch. In the following forcing question, the finite set $H$ is intended to be the set of all branches which have not been forced yet.

\begin{definition}
Let $c = (\sigma^0_I, \sigma^1_I, X_\nu, C  : I \lhd \Ical_n, \nu \in \Ical_n) \in \Pb_n$, let $H \subseteq \{ I \lhd \Ical_n \}$, let $\Phi_e(G, x, y)$ be a $\Delta_0$ formula with free variables $x$ and $y$ and let $i < 2$.
Define the relation $c \qvdash^i_H (\exists x)(\forall y)\Phi_e(G, x, y)$
to hold if
$$
\Ucal^{\Ical_n}_C \cap \bigcap \{ \Ical_n \odot \Ucal^I_{\zeta_I(e, \sigma^i_I \cup \rho, x)} : I \in H, x \in \omega, \rho \subseteq A^i \cap \bigcup_{\nu \in I} X_\nu \}
$$
is not a largeness class
\end{definition}

\begin{lemma}
Let $c \in \Pb_n$, let $H \subseteq \{ I \lhd \Ical_n \}$, let $\Phi_e(G, x, y)$ be a $\Delta_0$ formula with free variables $x$ and $y$, and let $i < 2$.
The relation $c \qvdash^i_H (\exists x)(\forall y)\Phi_e(G, x, y)$ is $\Sigma^{0, \emptyset'}_1$.
\end{lemma}
\begin{proof}
By \Cref{lem:decreasing-largeness-yields-largeness}, $c \qvdash^i_H (\exists x)(\forall y)\Phi_e(G, x, y)$ holds if there is a finite set $F \subseteq C$, and some $t \in \omega$ such that the following class
$$
\Ucal^{\Ical_n}_F \cap \bigcap \{ \Ical_n \odot \Ucal^I_{\zeta_I(e, \sigma^i_I \cup \rho, x)} : I \in H, x < t, \rho \subseteq A^i \cap \bigcup_{\nu \in I} X_\nu \uh t \}
$$
is not a largeness class. Note that this class is $\Sigma^{0,Z}_1$ for some $Z \in \Mcal$. By \Cref{lem:largeness-class-complexity},
not being a largeness class for a $\Sigma^{0,Z}_1$ class is $\Sigma^{0,Z}_2$, hence $\Sigma^0_1(\emptyset')$ whenever $Z$ is low. Thus, the whole formula is $\Sigma^0_1(A^i \oplus C \oplus \emptyset')$. Since $A^i$ and $C$ are $\Delta^0_2$, the formula is $\Sigma^0_1(\emptyset')$.
\end{proof}

The following lemma states that in the $\Sigma^0_2$ outcome, one can find an extension forcing the $\Sigma^0_2$ formula on one branch of $H$, which is the set of branches not having been satisfied yet.

\begin{lemma}\label{lem:forcing-question-sigma-case}
Let $c \in \Pb_n$, let $H \subseteq \{ I \lhd \Ical_n \}$, let $\Phi_e(G, x, y)$ be a $\Delta_0$ formula with free variables $x$ and $y$ and let $i < 2$.
Suppose
$$
c \qvdash^i_H (\exists x)(\forall y)\Phi_e(G, x, y)
$$
then there is some $d \in \Pb_n$ with $d \leq c$
and some $I \in H$ such that $d^{[I]} \Vdash^i (\exists x)(\forall y)\Phi_e(G, x, y)$.
\end{lemma}
\begin{proof}
Say $c = (\sigma^0_I, \sigma^1_I, X_\nu, C  : I \lhd \Ical_n, \nu \in \Ical_n)$.
Since $c \qvdash^i_H (\exists x)(\forall y)\Phi_e(G, x, y)$, then
by \Cref{lem:decreasing-largeness-yields-largeness}, there is a finite set $F \subseteq C$, and some $t \in \omega$ such that the following class
$$
\Ucal^{\Ical_n}_F \cap \bigcap \{ \Ical_n \odot \Ucal^I_{\zeta_I(e, \sigma^i_I \cup \rho, x)} : I \in H, x < t, \rho \subseteq A^i \cap \bigcup_{\nu \in I} X_\nu \uh t \}
$$
is not a largeness class. Since the class is $\Sigma^{0,Y}_1$ for some some $Y \in \Mcal$ and since $\Mcal$ is a Scott set, there is a cover $Z_0 \cup \dots \cup Z_{k-1} = \omega$ in $\Mcal$ such that for every $j < k$, $Z_j \not \in \Ucal^{\Ical_n}_F \cap \bigcap \{ \Ical_n \odot \Ucal^I_{\zeta_I(e, \sigma^i_I \cup \rho, x)} : I \in H, x < n, \rho \subseteq A^i \cap \bigcup_{\nu \in I} X_\nu \uh n \}$.

By \Cref{lem:infinite-intersection-largeness}, $\Lcal(\Ucal^{\Ical_n}_C)$ is a largeness class, then there is some $\langle j_\nu : \nu \in \Ical_n \rangle$ such that $\langle Z_{j_\nu} : \nu \in \Ical_n \rangle \in \Lcal(\Ucal^{\Ical_n}_C)$.
Thus $\Ucal^{\Ical_n}_C \cap \Lcal_{\langle X_\nu : \nu \in \Ical_n \rangle} \cap \Lcal_{\langle Z_{j_\nu} : \nu \in \Ical_n \rangle}$ is a largeness class,
so by \Cref{lem:intersection-compatibility}, the class $\Ucal^{\Ical_n}_C \cap  \Lcal_{\langle X_\nu \cap Z_{j_\nu} : \nu \in \Ical_n \rangle}$ is a largeness class.
In particular $\langle X_\nu \cap Z_{j_\nu} : \nu \in \Ical_n \rangle \in \Ucal^{\Ical_n}_C$, so there is some $I \in H$, some $x < t$ and some $\rho \subseteq A^i \cap \bigcup_{\nu \in I} X_\nu \uh t$ such that
$$
\langle X_\nu \cap Z_{j_\nu} : \nu \in \Ical_n \rangle \not \in \Ical_n \odot \Ucal^I_{\zeta_I(e, \sigma^i_I \cup \rho, x)}
$$
Let $D \supseteq C$ be such that $\Ucal^{\Ical_n}_D = \Ucal^{\Ical_n}_C \cap  \Lcal_{\langle X_\nu \cap Z_{j_\nu} : \nu \in \Ical_n \rangle}$.
For every $\nu \in \Ical_n$, let $Y_\nu : (X_\nu \cap Z_{j_\nu}) - \{0, \dots, t\}$.
In particular, $\Ucal^{\Ical_n}_D \subseteq \Lcal_{\langle Y_\nu : \nu \in \Ical_n \rangle}$.
Let $\tau^i_I = \sigma^i_I \cup \rho$, and $\tau^{1-i}_I = \sigma^{1-i}_I$.
For every $J \lhd \Ical_n$ with $J \neq I$, let $\tau^0_J = \sigma^0_J$ and $\tau^1_J = \sigma^1_J$. The $\Pb_n$-condition $d = (\tau^0_J, \tau^1_J, Y_\nu, D  : J \lhd \Ical_n, \nu \in \Ical_n)$ is an extension of $c$
such that $d^{[I]} \Vdash^i (\exists x)(\forall y)\Phi_e(G, x, y)$ with $I \in H$.
\end{proof}

The following lemma states that whenever sufficiently many formulas have satisfied the $\Pi^0_2$ outcome, then one can find an extension with more branches, such that any branch refining a branch in $H$ will force at least two of the $\Pi^0_2$ formulas. Letting $H$ be the set of branches for which the requirement has not been forced yet, one obtain an extension on which the requirement is forced on all branches simultaneously.

\begin{lemma}\label{lem:forcing-question-pi-case}
Let $c \in \Pb_n$, let $H \subseteq \{ I \lhd \Ical_n \}$, let $\Phi_{e_0}(G, x, y), \dots, \Phi_{e_{2u_n}}(G, x, y)$ be $2u_n+1$ many $\Delta_0$ formulas with free variables $x$ and $y$ and let $i < 2$.
Suppose that for every $j \leq 2u_n$,
$$
c \nqvdash^i_H (\exists x)(\forall y)\Phi_{e_j}(G, x, y)
$$
Then there is some $d \in \Pb_{n+1}$ with $d \leq c$
such that for every $I \in H$ and $J \lhd \Ical_{n+1}$ such that $J \leq I$,
there are some $a < b \leq 2u_n$ such that
$$
d^{[J]} \Vdash^i (\forall x)(\exists y)\neg \Phi_{e_a}(G, x, y)
\hspace{10pt}\mbox{ and }\hspace{10pt}
d^{[J]} \Vdash^i (\forall x)(\exists y)\neg \Phi_{e_b}(G, x, y)
$$
\end{lemma}
\begin{proof}
Say $c = (\sigma^0_I, \sigma^1_I, X_\nu, C  : I \lhd \Ical_n, \nu \in \Ical_n)$.
For every $j \leq 2u_n$, the class
$$
\Acal_j = \Ucal^{\Ical_n}_C \cap \bigcap \{ \Ical_n \odot \Ucal^I_{\zeta_I(e_j, \sigma^i_I \cup \rho, x)} : I \in H, x \in \omega, \rho \subseteq A^i \cap \bigcup_{\nu \in I} X_\nu \}
$$
is a largeness class.
Let $D \subseteq \omega$ be a $\Delta^0_2$ set such that $\Ucal^{\Ical_{n+1}}_D$ is the class of all $\langle Z_{j^\frown \nu} : j \leq 2u_n, \nu \in \Ical_n \rangle$ such that for every $j \leq 2u_n$, $\langle Z_{j^\frown \nu} : \nu \in \Ical_n\} \in \Acal_j$. In particular, $\Ucal^{\Ical_{n+1}}_D$ is a largeness class.
For every $j^\frown \nu \in \Ical_{n+1}$, let $Y_{j^\frown \nu} = X_\nu$. For every $J \lhd \Ical_{n+1}$, let $\tau^0_J = \sigma^0_I$ and $\tau^1_J = \sigma^1_I$, where $I \lhd \Ical_n$ is the unique $n$-index such that $J \leq I$.

\smallskip
\emph{Claim 1}: $\Ucal^{\Ical_{n+1}}_D \subseteq \Lcal_{\langle Y_\mu : \mu \in \Ical_{n+1}\rangle}$. Let $\langle Z_{j^\frown \nu} : j \leq 2u_n, \nu \in \Ical_n \rangle \in \Ucal^{\Ical_{n+1}}_D$. For every $j \leq 2u_n$, $\langle Z_{j^\frown \nu} : \nu \in \Ical_n\} \in \Acal_j$. Since $\Acal_j \subseteq \Lcal_{\langle X_\nu : \nu \in \Ical_n\rangle}$, then $|Z_{j^\frown \nu} \cap X_\nu| = \infty$. Since $X_\nu = Y_{j^\frown \nu}$, then $|Z_{j^\frown \nu} \cap Y_{j^\frown \nu}| = \infty$, so $\langle Z_{j^\frown \nu} : j \leq 2u_n, \nu \in \Ical_n \rangle \in \Lcal_{\langle Y_\mu : \mu \in \Ical_{n+1}\rangle}$. This proves Claim 1.

\smallskip
\emph{Claim 2}: $\Ucal^{\Ical_{n+1}}_D \leq \Ucal^{\Ical_n}_C$.
We need to prove that $\Ucal^{\Ical_{n+1}}_D \subseteq \Ical_{n+1} \otimes \Ucal^{\Ical_n}_C$. Fix $\langle Z_{j^\frown \nu} : j \leq 2u_n, \nu \in \Ical_n \rangle \in \Ucal^{\Ical_{n+1}}_D$. Then for every $j \leq 2u_n$, $\langle Z_{j^\frown \nu} : \nu \in \Ical_n\rangle \in \Acal_j \subseteq \Ucal^{\Ical_n}_C$. Thus $\Ucal^{\Ical_{n+1}}_D \subseteq \Ical_{n+1} \otimes \Ucal^{\Ical_n}_C$. This proves Claim 2.

Let $d = (\tau^0_J, \tau^1_J, Y_\mu, D  : J \lhd \Ical_{n+1}, \mu \in \Ical_{n+1})$. In particular $d$ is a $\Pb_{n+1}$-condition extending $c$.
Fix $I \in H$ and $J \lhd \Ical_{n+1}$ such that $J \leq I$. In particular, there are some $a < b \leq 2u_n$ such that $J = \{a^\frown \nu : \nu \in I\} \cup \{b^\frown \nu : \nu \in I\}$.

\smallskip
\emph{Claim 3}:
$d^{[J]} \Vdash^i (\forall x)(\exists y)\neg \Phi_{e_a}(G, x, y)$
and
$d^{[J]} \Vdash^i (\forall x)(\exists y)\neg \Phi_{e_b}(G, x, y)$.
We prove that $d^{[J]} \Vdash^i (\forall x)(\exists y)\neg \Phi_{e_a}(G, x, y)$. The other case is symmetric. For every $x \in \omega$ and $\rho \subseteq A^i \cap \bigcup_{\mu \in J} Y_\mu$,
in particular $\rho \subseteq A^i \cap \bigcup_{\nu \in I} X_\nu$. Fix $\langle Z_\mu : \mu \in J\rangle \in \pi_J(\Ucal^{\Ical_{n+1}}_D)$. In particular $\langle Z_{a^\frown \nu} : \nu \in I\rangle \in \Acal_a \subseteq \Ucal^I_{\zeta_I(e_a, \sigma^i_I \cup \rho, x)}$. So $(\sigma^i_I \cup \rho, \bigcup_{\nu \in I}  Z_{a^\frown \nu}) \not \Vdash (\forall y)\Phi_{e_a}(G, x, y)$. As $\sigma^i_I = \tau^i_J$ and $\bigcup_{\nu \in I} Z_{a^\frown \nu} \subseteq \bigcup_{\mu \in J} Z_\mu$, then $(\tau^i_J \cup \rho, \bigcup_{\mu \in J}  Z_\mu) \not \Vdash (\forall y)\Phi_{e_a}(G, x, y)$. So $\langle Z_\mu : \mu \in J\rangle \in \Ucal^J_{\zeta_J(e_a, \tau^i_J \cup \rho, x)}$.
Thus for every $x \in \omega$ and $\rho \subseteq A^i \cap \bigcup_{\mu \in J} Y_\mu$, $\pi_J(\Ucal^{\Ical_{n+1}}_D) \subseteq \Ucal^J_{\zeta_J(e_a, \tau^i_J \cup \rho, x)}$. This is the definition of $d^{[J]} \Vdash^i (\forall x)(\exists y)\neg \Phi_{e_a}(G, x, y)$. This proves Claim 3 and \Cref{lem:forcing-question-pi-case}.
\end{proof}

\subsection{Requirements}

We now define the requirements specific to our purpose, namely, obtaining a set whose jump does not compute a $\{0,1\}$-valued completion of the partial function $n \mapsto \Phi^{\emptyset'}_n(n)$.

\begin{definition}Fix a $\Delta_0$ formula $\Phi_e(G, x, y, p)$ with free integer variables $x$ and $y$, and free valuation variable $p$.
\begin{itemize}
	\item[(1)] Let $c \in \Qb_n$ and $i < 2$. We say that \emph{$c$ forces the $e$-th requirement on side $i$} if $c \Vdash^i (\exists x)(\forall y)\Phi_e(G, x, y, p)$ for some $\emptyset'$-correct valuation $p$, or $c \Vdash^i (\forall x)(\exists y)\neg \Phi_e(G, x, y, p_0)$ and $c \Vdash^i (\forall x)(\exists y)\neg \Phi_e(G, x, y, p_1)$ for two incompatible valuations.
	\item[(2)] Let $c \in \Pb_n$ and $i < 2$. We say that  \emph{$c$ forces the $e$-th requirement on side $i$} if $c^{[I]}$ forces the $e$-th requirement on side $i$ for every $I \lhd \Ical_n$.
\end{itemize}
\end{definition}

Given a condition $c \in \Pb_n$, $e \in \omega$ and $i < 2$, let $H(c, e, i)$ be the set of $I \lhd \Ical_n$ such that $c$ does not force the $e$-th requirement on the $i$-th side.

\begin{lemma}\label{lem:satisfy-requirement-one-side}
For every $c \in \Pb_n$, $i < 2$ and $e \in \omega$ such that $H(c, e, i) \neq \emptyset$, there is $\Pb$-condition $d \leq c$ such that $|H(d, e, i)| < |H(c, e, i)|$.
\end{lemma}
\begin{proof}
Let $H = H(c, e, i)$ and $W$ be the set of all valuations $p$ such that $c \qvdash^i_H (\exists x)(\forall y)\Phi_e(G, x, y, p)$. By \Cref{lem:largeness-class-complexity}, the set $W$ is $\emptyset'$-c.e, so by \Cref{lem:combi-liu-valuation}, we have two cases.

Case 1: $p \in W$ for some $\emptyset'$-correct valuation $p$.
By definition of $W$, $c \qvdash^i_H (\exists x)(\forall y)\Phi_e(G, x, y, p)$.
By \Cref{lem:forcing-question-sigma-case}, there is a $\Pb_n$-condition $d \leq c$ such that $|H(d, e, i)| < |H(c, e, i)|$.

Case 2: $p_0, \dots, p_{2u_n} \not \in W$ for $2u_n+1$ pairwise incompatible valuations. So
$$
c \nqvdash^i_H (\exists x)(\forall y)\Phi_e(G, x, y, p_j)
$$ for every $j \leq 2u_n$. By \Cref{lem:forcing-question-pi-case}, there is a $\Pb_{n+1}$-condition $d \leq c$ such that $d^{[J]}$ forces the $e$-th requirement on side $i$ for every $J \lhd \Ical_n$ such that $J \leq I$ for some $I \in H = H(c, e, i)$. Therefore $|H(d, e, i)| = 0 < |H(c, e, i)|$.
\end{proof}

\begin{lemma}\label{lem:force-requirements-everywhere}
For every $c \in \Pb$ and $e \in \omega$, there is $\Pb$-condition $d \leq c$ forcing the $e$-th requirement on both sides.
\end{lemma}
\begin{proof}
Apply iteratively \Cref{lem:satisfy-requirement-one-side} to obtain a condition $d_0$ such that $H(d_0, e, 0) = \emptyset$. Then apply again iteratively \Cref{lem:satisfy-requirement-one-side} below $d_0$ to obtain an extension $d_1$ such that $H(d_1, e, 1) = \emptyset$. The condition $d_1$ is the desired extension.
\end{proof}

\subsection{Construction}

As explained, a $\Pb$-condition represents multiple parallel $\Qb$-conditions.
By \Cref{lem:every-condition-valid-side}, every $\Pb$-condition admits a branch with a valid side. Moreover, by \Cref{lem:valid-upward-closed}, the valid sides of $\Qb$-conditions are upward-closed under the extension relation. This motivates the following definition.

\begin{definition}
A \emph{path} through a $\Pb$-filter $\Fcal$ is a pair $\langle P, i \rangle$
where $i < 2$ and for every $n \in \omega$, $P(n) \lhd \Ical_n$ is such that $P(n+1) \leq P(n)$ and for every $c \in \Fcal \cap \Pb_n$, $c^{[P(n)]}$ is $i$-valid.
\end{definition}

By \Cref{lem:every-condition-valid-side} and \Cref{lem:valid-upward-closed},
every $\Pb$-filter admits a path. We then let
$$
\Fcal(P, i) = \bigcup \{ \sigma^i_{P(n)} : (\sigma^0_I, \sigma^1_I, X_\nu, C  : I \lhd \Ical_n, \nu \in \Ical_n) \in \Fcal \}
$$
We can prove that the forced formulas hold along any path.

\begin{lemma}\label{lem:forced-formula-holds}
Let $\Fcal$ be a sufficiently generic $\Pb$-filter, and let $\langle P, i \rangle$ be a path through $\Fcal$ and let $G^i = \Fcal(P, i)$.
Let $\Phi_e(G, x, y)$ be a $\Delta_0$ formula and $c \in \Fcal$.
\begin{itemize}
	\item[(1)] If $c^{[P(n)]} \Vdash^i (\exists x)(\forall y)\Phi_e(G, x, y)$, then $(\exists x)(\forall y)\Phi_e(G^i, x, y)$ holds.
	\item[(2)] If $c^{[P(n)]} \Vdash^i (\forall x)(\exists y)\neg \Phi_e(G, x, y)$, then $(\forall x)(\exists y)\neg \Phi_e(G^i, x, y)$ holds.
\end{itemize}
\end{lemma}
\begin{proof}
Say $c^{[P(n)]} = (\sigma^0, \sigma^1, X_\nu, \Acal  : \nu \in P(n)) \in \Qb_n$.

(1) By definition of $c^{[P(n)]} \Vdash^i (\exists x)(\forall y)\Phi_e(G, x, y)$, then there is some $x \in \omega$ such that $(\sigma^i, \bigcup_{\nu \in P(n)} X_\nu) \Vdash (\forall y)(G^i, x, y)$. In particular, $\sigma^i \prec G^i$ and $G^i - \sigma^i \subseteq \bigcup_{\nu \in P(n)} X_\nu$, so for every $y \in \omega$, $\neg \Phi(G^i, x, y)$ holds.

(2) By \Cref{lem:forcing-relation-extension}, \Cref{lem:valid-side-progress-pi02} and \Cref{lem:compatibility-extension-qb-pb}, for every $x \in \omega$, there is some $m \in \omega$ and $d \in \Fcal \cap \Pb_m$ such that $d^{[P(m)]} = (\tau^0, \tau^1, Y_\mu, \Bcal  : \mu \in P(m))$ and $(\tau^i, \bigcup_{\mu \in P(m)} Y_\mu) \Vdash (\exists y)\neg \Phi_e(G, x, y)$. In particular, $\tau^i \prec G^i$, so $(\exists y)\Phi_e(G^i, x, y)$ holds.
\end{proof}

We are now ready to prove \Cref{thm:jump-pa-avpoidance-d22}.

\begin{proof}[Proof of \Cref{thm:jump-pa-avpoidance-d22}]
We prove the theorem for $Z = \emptyset$, as the whole argument relativizes.
Fix a $\Delta^0_2$ set $A$ and let $A^0 = A$ and $A^1 = \overline{A}$.
Let $\Fcal$ be a sufficiently generic $\Pb$-filter. Let $\langle P, i\rangle$ be a path through $\Fcal$. Let $G^i = \Fcal(P, i)$. By definition of a $\Pb$-condition, $G^i \subseteq A^i$. By \Cref{lem:valid-side-infinite} and \Cref{lem:compatibility-extension-qb-pb}, $G^i$ is infinite.

By \Cref{lem:force-requirements-everywhere}, for every $e \in \omega$, there is some $n \in \omega$ and some $c \in \Fcal \cap \Pb_n$ such that $c$ forces the $e$-th requirement both sides. In particular, $c^{[P(n)]}$ forces the $e$-th requirement on side~$i$. By \Cref{lem:forced-formula-holds}, the $e$-th requirement holds on $G^i$. By \Cref{lem:jump-pa-reformulated}, the jump of $G^i$ is not PA over~$\emptyset'$. This completes the proof of \Cref{thm:jump-pa-avpoidance-d22}.
\end{proof}

%
%


\vspace{0.5cm}

\bibliographystyle{plain}
\bibliography{bibliography}

\end{document}